\documentclass[preprint,12pt,authoryear]{elsarticle}
\usepackage{amscd,amssymb,stmaryrd}
\usepackage{amsmath}
\usepackage{amssymb}
\usepackage{graphicx}
\usepackage{subcaption}
\usepackage[utf8]{inputenc}
\usepackage[export]{adjustbox}
\usepackage{wrapfig}
\usepackage{amstext}
\usepackage{pstricks,pst-node,pst-plot,pst-coil}
\usepackage{amsthm}
\usepackage{amsmath}
\usepackage{color}
\usepackage{mathtools}
\usepackage{hyperref}
\usepackage[english]{babel}
\usepackage{booktabs}
\usepackage{mathrsfs}
\usepackage{comment}
\usepackage{float}
\usepackage[linesnumbered,ruled,vlined, ruled,vlined]{algorithm2e}
\usepackage{tikz}

\theoremstyle{definition}
\newtheorem{theorem}{Theorem}[section]

\newtheorem{lemma}[theorem]{Lemma}

\theoremstyle{remark}
\newtheorem*{remark}{Remark}

\def\eg{e.g., }
\def\ie{i.e., }
\def\etal{et al. }

\journal{Engineering Applications of Artificial Intelligence }
\begin{document}
\begin{frontmatter}
\title{Learning the Dynamical Response of Nonlinear Non-Autonomous Dynamical Systems with Deep Operator Learning Neural Networks}
\author{Guang Lin \fnref{label1,label2}}
\ead{guanglin@purdue.edu}
\author{Christian Moya \fnref{label2}}
\ead{cmoyacal@purdue.edu}
\author{Zecheng Zhang \fnref{label2}}
\ead{zecheng.zhang.math@gmail.com}

\affiliation[label1]{organization={Department of Mechanical Engineering, Purdue University},
            addressline={610 Purdue Mall},
            city={West Lafayette},
            postcode={47907},
            state={IN},
            country={USA}}
\affiliation[label2]{organization={Department of Mathematics, Purdue University},
            addressline={610 Purdue Mall},
            city={West Lafayette},
            postcode={47907},
            state={IN},
            country={USA}}
\begin{abstract}
We propose using operator learning to approximate the dynamical system response with control, such as nonlinear control systems. Unlike classical function learning, operator learning maps between two function spaces, does not require discretization of the output function, and provides flexibility in data preparation and solution prediction. Particularly, we apply and redesign the Deep Operator Neural Network (DeepONet) to recursively learn the solution trajectories of the dynamical systems. Our approach involves constructing and training a DeepONet that approximates the system's local solution operator. We then develop a numerical scheme that recursively simulates the system's long/medium-term dynamic response for given inputs and initial conditions, using the trained DeepONet. We accompany the proposed scheme with an estimate for the error bound of the associated cumulative error. Moreover, we propose a data-driven Runge-Kutta (RK) explicit scheme that leverages the DeepONet's forward pass and automatic differentiation to better approximate the system's response when the numerical scheme's step size is small. Numerical experiments on the predator-prey, pendulum, and cart pole systems demonstrate that our proposed DeepONet framework effectively learns to approximate the dynamical response of non-autonomous systems with time-dependent inputs.  
\end{abstract}
\end{frontmatter}
\section{Introduction} \label{sec:introduction}
High-fidelity numerical schemes are the primary computational tools for simulating and predicting complex dynamical systems, such as climate modeling, robotics, or the modern power grid. However, these schemes may be too expensive for control, optimization, and uncertainty quantification tasks, which often require a large number of forward simulations and consume significant computational resources. Therefore, there is a growing interest in developing tools that can accelerate the numerical simulation of complex dynamical systems without compromising accuracy.

By providing faster alternatives to traditional numerical schemes, machine learning-based computational tools hold the promise of accelerating the rate of innovation for complex dynamical systems. {Hence, a recent wave of machine and deep learning-based tools has demonstrated the potential of using observational data to construct fast surrogates of complex systems~\cite{efendiev2021hei, qin2021data} and providing efficient solutions to complex engineering tasks such as fault diagnosis~\cite{choudhary2023multi,mishra2022self,mishra2022fault} and power engineering~\cite{moya2023dae,moya2023deeponet}.} These tools aim to (i) learn the governing equations of a dynamical system or (ii) learn to predict the system's response from data.

On the one hand, several works~\cite{brunton2016discovering, brunton2016sparse, schaeffer2017learning, sun2020neupde} use observational data to discover the unknown governing equations of the underlying system. For example, Brunton \etal~\cite{brunton2016discovering} proposed identifying nonlinear systems using sparse schemes and a large set of dictionaries. They then extended their work in~\cite{brunton2016sparse} to identify input/output mappings that describe control systems. In~\cite{sun2020neupde}, the authors used sparse approximation schemes to recover the governing partial or ordinary differential equations of unknown systems.

On the other hand, there is growing interest in predicting the future response of a dynamical system using time-series data~\cite{qin2019data, raissi2018multistep, qin2021data, proctor2018generalizing}. For instance, the authors of~\cite{efendiev2021hei, chung2021multi} trained a transformer with data from the early stages of time-dependent partial differential equations to recursively predict the solution of future stages. Qin \etal~\cite{qin2019data} used a recurrent residual neural network~(ResNet) to approximate the mapping from the current state to the next state of an unknown autonomous system. Similarly, in~\cite{raissi2018multistep}, the authors used feed-forward neural networks~(FNN) as a building block of a multi-step scheme that predicts the response of an autonomous system. In~\cite{qin2021data}, the authors extended their previous work~\cite{qin2019data} to non-autonomous systems with time-dependent inputs. To this end, they parametrized the input locally within small time intervals.

Many of the works mentioned above require large amounts of training data to avoid overfitting. However, acquiring this data may be prohibitively expensive for complex dynamical systems. Additionally, if one fails to avoid overfitting when using traditional neural networks for next-state methods, the predicted response may drift and accumulate errors over time. It is therefore crucial to develop deep learning-based frameworks that can efficiently handle the infinite-dimensional nature of predicting the response of non-autonomous systems with time-dependent inputs for long time horizons, as studied in this paper.

\textit{Operator Learning}~\cite{chen1995universal, lu2019deeponet, li2020fourier, zhang2022belnet} was first addressed by the seminal work of \cite{chen1995universal}. Unlike classical function learning, operator learning involves approximating an operator between two function spaces. Recently, the seminal paper~\cite{lu2021learning} extended the work of \cite{chen1995universal} and proposed the Deep Operator Neural Network~(DeepONet) framework. DeepONet exhibits small generalization errors and learns with limited training data \cite{zhang2022belnet}. This has been demonstrated in many application areas, such as power engineering~\cite{moya2023deeponet}, multi-physics problems~\cite{cai2021deepm}, and turbulent combustion problems~\cite{ranade2021generalized}. Extensions to the original DeepONet~\cite{lu2021learning} have enabled the incorporation of physics-informed models~\cite{wang2021learning}, handling noisy data~\cite{lin2023b}, and designing novel optimization methods~\cite{lin2023b,lin2021multi}.

DeepONet is a prediction-free and domain-free tool \cite{zhang2022belnet}. Specifically, DeepONet can predict the target function value at any point in its domain, and the input and output functions are not necessarily required to share the same domain. This property allows for handling incomplete datasets during the training phase. In \cite{zhang2022belnet}, it was further relaxed the assumptions on input function discretization, improving flexibility in data preparation and prediction accuracy for operator learning.
 
In this paper, we focus on extending the This paper focuses on extending the original DeepONet framework to learn the solution operator of non-autonomous systems with time-dependent inputs for long-time horizons. By using such a data-driven operator framework, control policies can be designed for continuous nonlinear systems, among other applications.

In particular, one of our motivations behind learning to approximate the solution operator of a non-autonomous system is its application in control, and Model-Based Reinforcement Learning (MBRL)~\cite{wang2019benchmarking,sutton1990integrated}. In MBRL, one learns to approximate the system to control (from data) and then uses the learned model to seek an optimal policy without extensive interaction with the actual system. A common approach is to learn a discrete-time forward model that predicts the next state using the current state and selected control action. Such an MBRL framework has delivered successful and efficient results in discrete-time problems such as games~\cite{kaiser2019model}. However, most control systems in science and engineering (\eg robotics~\cite{nagabandi2020deep}, unmanned vehicles~\cite{hafner2019dream}, or laminar flows~\cite{fan2020reinforcement}) are continuous. Of course, one can always discretize the continuous dynamics and apply discrete-time MBRL using traditional neural networks (\eg see~\cite{brockman2016openai}). However, if one fails to handle the inherent epistemic uncertainty~\cite{deisenroth2013gaussian}, such a strategy may lead to error accumulation (due to model bias) and poor asymptotic performance. To alleviate these drawbacks, we build on the original DeepONet~\cite{lu2021learning} to design an effective and efficient framework that can learn the solution operator of a nonlinear non-autonomous system with time-dependent inputs, which one can then apply within the framework of control or continuous MBRL~\cite{du2020model}.

Formally, the objectives of this paper are twofold. 
\begin{enumerate}
    \item \textit{Approximation of the local solution operator:} Our goal is to create a neural operator-based framework that can learn to map (i) the current state of the dynamical system with control and (ii) a local approximation of the input to the next state of the non-autonomous dynamical system. 
    \item \textit{Long/Medium-term simulation:} We aim to design an efficient scheme that uses the operator learning framework to simulate the dynamical system's response to a given input over a long or medium-term horizon.
\end{enumerate}
The contributions that will achieve the above objectives are summarized below.
\begin{enumerate}
    \item We model the non-autonomous/control dynamical system in the deep operator learning framework and propose a DeepONet-based numerical scheme that effectively and recursively simulates the solution of the system.
    \item We also propose a novel data-driven Runge-Kutta (RK) DeepONet-based method that reduces the cumulative error and improves recursive prediction accuracy in long-term simulations. 
    \item We provide an estimation of the cumulative error for the proposed numerical scheme based on DeepONet. Our estimated error bounds are tighter than those presented in \cite{qin2021data}. Additionally, we demonstrate that the novel data-driven RK DeepONet-based scheme achieves better stability bounds than the DeepONet-based scheme.
    \item We test the proposed frameworks on state-of-the-art models, such as the predator-prey system, the pendulum, the cart-pole system, and a power engineering task. The effectiveness of the methods is observed in all experiments.
\end{enumerate}
We organize the rest of this paper as follows. In Section~\ref{sec:problem-formulation}, we review operator learning and introduce the proposed DeepONet-based algorithms. Next, we present the novel data-driven Runge-Kutta DeepONet-based scheme and estimate its corresponding error bound in Section~\ref{sec:data-driven-rk}. We present the numerical experiments in Section~\ref{sec:numerical-experiments}. We provide a discussion of our results and future work in Section~\ref{sec:discussion}, and conclude the paper in Section~\ref{sec:conclusion}.
\section{Problem Formulation} \label{sec:problem-formulation}
We consider the problem of learning from data the solution operator of the following continuous-time non-autonomous system with time-dependent inputs
\begin{gather} 
\begin{aligned} 
\frac{d}{dt}x(t) &= f(x(t), u(t)), \quad t \in [a,b] \\
x(a) &= x_0, \label{eq:control-system}
\end{aligned}
\end{gather}
where $x(t) \in \mathcal{X} \subseteq \mathbb{R}^n$ is the \textit{state} vector, $u(t) \in \mathcal{U} \subseteq \mathbb{R}^p$ the \textit{input} vector, and $f:\mathcal{X} \times \mathcal{U} \to \mathcal{X}$ an \textit{unknown} function. Additional assumptions on the input \textit{function}~u and the function~$f$ will be discussed later. Also, throughout this paper, we assume~$u(t)$ is a scalar input, \ie $p=1$. Extending the proposed framework to the vector-valued case is straightforward.

\textit{Solution Operator.} Let~$\mathcal{F}$ denote the solution operator of~\eqref{eq:control-system}. $\mathcal{F}$ takes as input the initial condition~$x(a) = x_0 \in \mathcal{X}$ and the sequence of input function values $u_{[a,t)}:=\{u(\tau) \in \mathcal{U}: \tau \in [a,t)\}$, and outputs the state~$x(t) \in \mathcal{X}$ at time~$t \in [a,b]$. We compute the solution operator via
\begin{gather} \label{eq:solution-operator}
\mathcal{F}\left(x_0, u_{[a,t)}\right)(t) \equiv x(t) = x_0 + \int_{a}^t f(x(s), u(s))ds.   
\end{gather}

\textit{Approximate System.} In practice, to learn the operator $\mathcal{F}$, we only have access to an approximate/discretized representation of the input function $u(t)$. Let~$u_m$ denote this approximate representation, which yields the following approximate system
\begin{gather}
\begin{aligned} \label{eq:approximate-system}
    \frac{d}{dt}\tilde{x}(t) &= f(\tilde{x}(t), u_m), \qquad t \in [a,b] \\
    \tilde{x}(a) &= x(a),
\end{aligned}    
\end{gather}
whose solution operator is
\begin{align} \label{eq:approximate-solution-operator}
    \mathcal{F}\left(x_0,u_m\right) (t) \equiv \tilde{x}(t) = x_0 + \int_a^t f(\tilde{x}(s), u_m)ds.
\end{align}
In the above, with a slight abuse of notation, we denoted $u_m$ as the input discretized using $m \ge 1$ sensors or interpolation points within the interval $[a,b]$, and the \textit{approximate} state function as $\tilde{x}(t)$.

\begin{remark}
In practice, the approximate system~\eqref{eq:approximate-system} with the solution operator~\eqref{eq:approximate-solution-operator} can represent, for example, sampled-data control systems~\cite{wittenmark2002computer} or semi-Markov Decision Processes~\cite{du2020model}. Therefore, the methods introduced in this paper can be used to design optimal control policies within the framework of model-based reinforcement learning~\cite{wang2019benchmarking}.
\end{remark}

\subsection{Learning the Solution Operator} \label{ssec:learning-operator}
To learn the solution operator $\mathcal{F}$, we use the \textit{Deep Operator Network}~(DeepONet) framework introduced in~\cite{lu2021learning}. In~\cite{lu2021learning}, the authors used a DeepONet~$G_\theta$ to learn a simplified version of the solution operator~$\mathcal{F}$~\eqref{eq:solution-operator} with $x(a) = 0$, \ie
$$
G(u)(t) \equiv x(t) = \int_{a}^t f(x(s), u(s))ds, \qquad t \in [a,b].
$$
The DeepONet~$G_\theta$ takes as inputs the (i) input $u(t)$ discretized using $m$ interpolation points (known as sensors in~\cite{lu2019deeponet}) and (ii) time~$t \in [a,b]$, and outputs the state~$x(t)$. Clearly, this DeepONet prediction is one-shot; it requires knowledge of the input in the whole interval~$[a,b]$. For small values of~$b$, the DeepONet's~$G_{\theta^*}$ prediction is very accurate. However, as we increase~$b$, the accuracy deteriorates. To improve accuracy, one can increase the number of interpolation points. This, however, makes DeepONet's training more challenging. 

To alleviate this drawback, we take a different approach in this paper. First, we train a DeepONet~$\mathcal{F}_\theta$, with the vector of trainable parameters~$\theta$, to learn the \textit{local} solution operator of the system. Then, we design a DeepONet-based numerical scheme that recursively uses the trained DeepONet~$\mathcal{F}_{\theta^*}$ to predict long/medium-term horizons, \ie for large values of $b$.

\textit{Learning the Local Solution Operator.} We let~$\mathcal{P}$ denote the possibly \textit{irregular} and \textit{arbitrary} time partition
$$
\mathcal{P}:~a=t_0 < t_1< \ldots < t_M=b,
$$
where $h_n:=t_{n+1} - t_n$ for all $n=0, 1,\ldots,M-1$ and let $h:= \max_{n} h_n$. Then, within the local interval~$[t_n,t_{n+1}] \equiv [t_n,t_n + h_n]$, the solution operator is given by
\begin{gather}
\begin{aligned} \label{eq:local-solution-operator}
    \mathcal{F}(\tilde{x}(t_n), u^n_m)(h_n) &\equiv \tilde{x}(t_n+h_n) \\ 
    &= \tilde{x}(t_n) + \int_{t_n}^{t_{n} + h_n} f(\tilde{x}(s), u^n_m)ds. 
\end{aligned}    
\end{gather}
In the above, with a slight abuse of notation, we use $u^n_m$ to denote the \textit{local} discretized representation of the input function~$u_m$, within the interval $[t_n, t_{n+1}]$, using $n_s \ge 1$ interpolation/sensor points or basis. 

\textit{The DeepONet $\mathcal{F}_\theta$.} Next, we design a Deep Operator Network (DeepONet) $\mathcal{F}_\theta$, with the vector of trainable parameters $\theta$, to approximate the local solution operator \eqref{eq:local-solution-operator}. Figure \ref{fig:deeponet} illustrates the proposed DeepONet~$\mathcal{F}_\theta$ that has two neural networks: the Branch Net and the Trunk Net.

\begin{figure}[t!]
\centering
\includegraphics[width=0.90\textwidth, height=6.5cm]{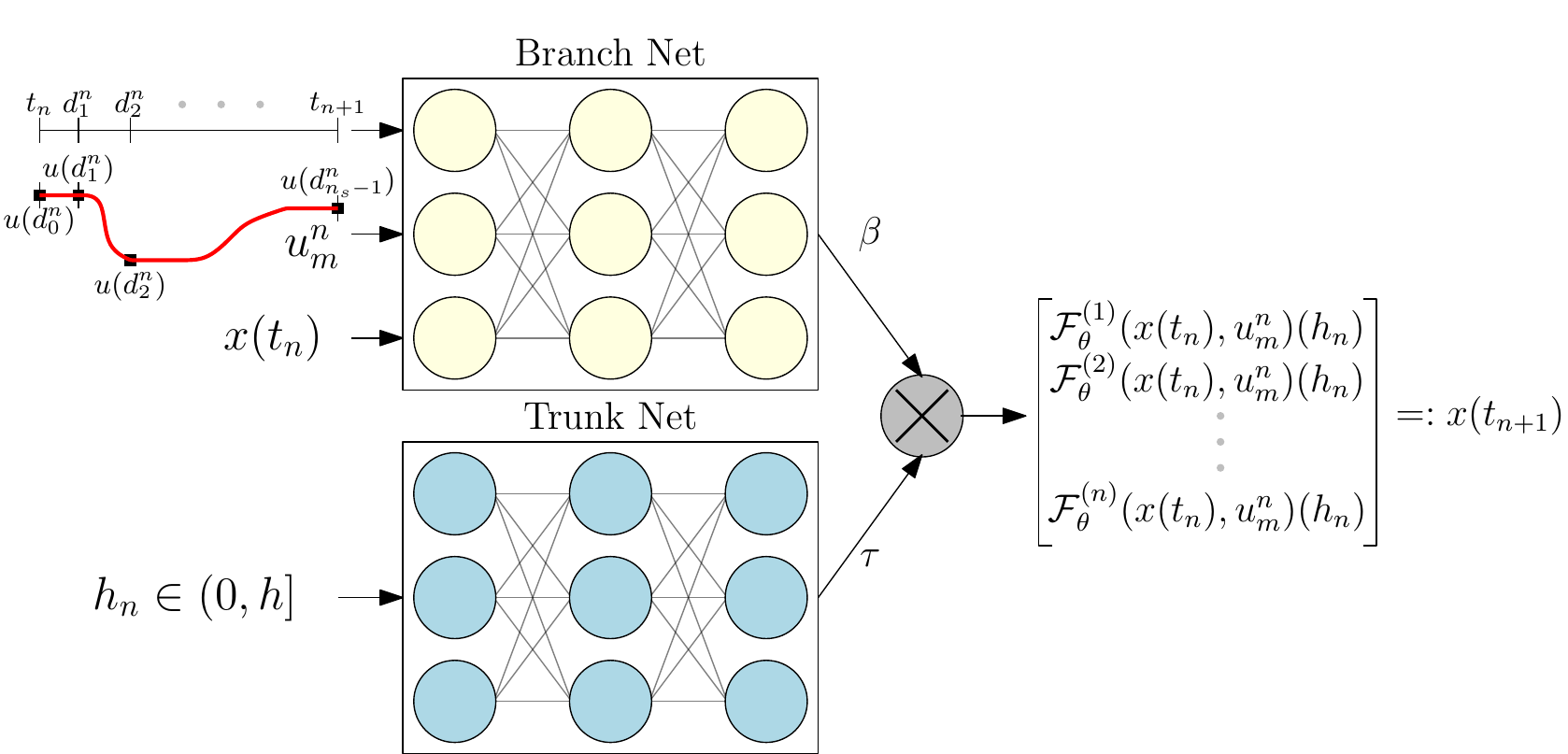}
\caption{The DeepONet architecture for learning the solution operator of non-autonomous system with time-dependent inputs. The Branch Net takes as input a vector resulting from concatenating (i) the current state~$\tilde{x}(t_n) \in \mathbb{R}^n$, (ii) the discretized local input~$u_m^n \in \mathbb{R}^{n_s}$, and (iii) the relative/flexible sensor locations $(t_n-d_0^n, \ldots, t_n-d_{{n_s}-1}^n)^\top \in \mathbb{R}^{n_s}$. Then, the Branch Net outputs the vector of features~$\beta \in \mathbb{R}^{nq}$. The Trunk Net takes as the input $h_n \in (0,h]$ and outputs the vector of features~$\tau \in \mathbb{R}^{nq}$. Finally, we obtain the DeepONet using the dot product between~$\beta$ and $\tau$.}
\label{fig:deeponet}
\end{figure}

The \textit{Branch Net} maps the vector that concatenates the (i) current state~$\tilde{x}(t_n)$ and (ii) input function~$u^n_m \in \mathbb{R}^{n_s}$, discretized using the mesh of \textit{local} sensors $t_n = d_0^n<d_1^n<...<d_{n_s-1}^n = t_{n+1}$, to the branch output feature vector~$\beta \in \mathbb{R}^{nq}$. Meanwhile, the \textit{Trunk Net} maps the scalar step size~$h_n \in (0,h]$ to the trunk output feature vector~$\tau \in \mathbb{R}^{nq}$. We compute the DeepONet's output for the $i$th component of the state vector~$\tilde{x}(t_{n+1})$ using the dot product:
$$\mathcal{F}^{(i)}_\theta(\tilde{x}(t_n),u^n_m)(h_n) = \sum_{k=1}^q \beta_{(i-1)q+k} \tau_{(i-1)q+k}.$$

\begin{remark}
\textit{Sensor locations.} One of the problems with the original DeepONet~\cite{lu2019deeponet} is that the sensor locations, $t_n = d_0^n<d_1^n<...<d_{n_s-1}^n = t_{n+1}$, are fixed. If we fix these sensor locations, then we cannot predict the response of the system using the arbitrary and irregular partition~$\mathcal{P}$. To enable prediction with~$\mathcal{P}$, we let the input to the branch~$u^n_m$ be the concatenation of the discretized input~$\left(u(d_0^n), \ldots, u(d_{{n_s}-1}^n)\right)$ with the corresponding (flexible) relative sensor locations~$(t_n-d_0^n, \ldots, t_n-d_{{n_s}-1}^n)$.
\end{remark}

\begin{remark}
\textit{The case of $n_s=1$ sensors, \ie the piece-wise constant approximation of~$u$.} If we let $n_s=1$ sensor, then the discretized input function within the interval $[t_n, t_{n+1}]$ corresponds to the singleton $u^n_m \equiv u(t_n)$. Such a case is the most challenging to learn because it introduces the largest error that propagates over time. However, it also represents one of the target applications: continuous model-based reinforcement learning with semi-Markov decision processes. We will show (in Section~\ref{sec:numerical-experiments}) that the proposed DeepONet can effectively handle having only $n_s=1$ sensor.
\end{remark}

\textit{Training the DeepONet}~$\mathcal{F}_\theta$. We train the proposed DeepONet~$\mathcal{F}_\theta$ model by minimizing the loss function 
$$\mathcal{L}(\theta;\mathcal{D}) = \frac{1}{N} \sum_{i=1}^N ||x^i(t_n +h_n^i) - \mathcal{F}_\theta(x^i(t_n), u^{i,n}_m)(h^i_{n}) ||_2^2$$
over of $N$ training data triplets 
$$ \mathcal{D} = \{(x^{i}(t_n), u^{i,n}_m), h_n^{i},x^{i}(t_n + h^{i}_n)\}_{i=1}^{N}$$
generated by the unknown ground truth local solution operator~$\mathcal{F}$.

\subsection{Predicting the System's Response} \label{ssec:deeponet-prediction}
We predict the response of the system~\eqref{eq:control-system} over a long/medium-term horizon (\ie within the interval $[a,b]$, with $b \gg 1$) using the DeepONet-based numerical scheme detailed in Algorithm~\ref{alg:prediction-scheme} {and illustrated in Figure~\ref{fig:recursive-DeepONet}.} Algorithm~\ref{alg:prediction-scheme} takes as inputs the (i) initial condition~$x(a) = x_0$, (ii) partition~$\mathcal{P}$, (iii) discretized representation of the input $u^n_m$, for $n=1,\ldots,M-1$, and (iv) trained DeepONet~$\mathcal{F}_{\theta{^*}}$. Then, Algorithm~\ref{alg:prediction-scheme} outputs the predicted response of the system over the partition~$\mathcal{P}$, \ie $\{\check{x}_n \equiv \check{x}(t_n):t_n \in \mathcal{P}\}$. Let us conclude this section by estimating a bound for the cumulative error of the proposed DeepONet-based numerical scheme described in Algorithm~\ref{alg:prediction-scheme}.
\begin{algorithm}[t]
\DontPrintSemicolon
\SetAlgoLined
\textbf{Require:} initial state vector~$x(a) = x_0$, partition~$\mathcal{P}$, input~$u^n_m$, for $n=1,\ldots,M-1$, and trained DeepONet~$\mathcal{F}_{\theta^*}$.\;
initialize $\check{x}(t_0) = x_0$\;
\For{$n = 0,\ldots,M-1$}{
  update the independent variable~$t_{n+1} = t_n + h_n$\;
  update the state vector using the DeepONet's forward pass
  $$\check{x}(t_{n+1}) = \mathcal{F}_{\theta^*}(\check{x}(t_n), u^n_m)(h_n).$$
  \vspace{-1.5em}\;}
  \textbf{Return:} predicted response~$\{\check{x}_n \equiv \check{x}(t_n):t_n \in \mathcal{P}\}$.\;
 \caption{DeepONet-based Numerical Scheme}
 \label{alg:prediction-scheme}
\end{algorithm}
\tikzset{every picture/.style={line width=0.75pt}} 
\begin{figure}[t!]
\centering
\begin{tikzpicture}[x=0.75pt,y=0.75pt,yscale=-1,xscale=1]
\draw  [fill={rgb, 255:red, 184; green, 233; blue, 134 }  ,fill opacity=1 ] (180,96) -- (351,96) -- (351,144) -- (180,144) -- cycle ;
\draw    (353,120) -- (385,120) ;
\draw [shift={(387,120)}, rotate = 180] [color={rgb, 255:red, 0; green, 0; blue, 0 }  ][line width=0.75]    (10.93,-3.29) .. controls (6.95,-1.4) and (3.31,-0.3) .. (0,0) .. controls (3.31,0.3) and (6.95,1.4) .. (10.93,3.29)   ;
\draw  [fill={rgb, 255:red, 241; green, 182; blue, 188 }  ,fill opacity=1 ] (387,120) .. controls (387,106.19) and (404.91,95) .. (427,95) .. controls (449.09,95) and (467,106.19) .. (467,120) .. controls (467,133.81) and (449.09,145) .. (427,145) .. controls (404.91,145) and (387,133.81) .. (387,120) -- cycle ;
\draw  [fill={rgb, 255:red, 241; green, 182; blue, 188 }  ,fill opacity=1 ] (21,59) .. controls (21,45.19) and (38.91,34) .. (61,34) .. controls (83.09,34) and (101,45.19) .. (101,59) .. controls (101,72.81) and (83.09,84) .. (61,84) .. controls (38.91,84) and (21,72.81) .. (21,59) -- cycle ;
\draw    (61,10) -- (61,32) ;
\draw [shift={(61,34)}, rotate = 270] [color={rgb, 255:red, 0; green, 0; blue, 0 }  ][line width=0.75]    (10.93,-3.29) .. controls (6.95,-1.4) and (3.31,-0.3) .. (0,0) .. controls (3.31,0.3) and (6.95,1.4) .. (10.93,3.29)   ;
\draw    (61,10) -- (427,10) ;
\draw    (427,10) -- (427,95) ;
\draw  [fill={rgb, 255:red, 241; green, 182; blue, 188 }  ,fill opacity=1 ] (21,121) .. controls (21,107.19) and (38.91,96) .. (61,96) .. controls (83.09,96) and (101,107.19) .. (101,121) .. controls (101,134.81) and (83.09,146) .. (61,146) .. controls (38.91,146) and (21,134.81) .. (21,121) -- cycle ;
\draw  [fill={rgb, 255:red, 241; green, 182; blue, 188 }  ,fill opacity=1 ] (21,184) .. controls (21,170.19) and (38.91,159) .. (61,159) .. controls (83.09,159) and (101,170.19) .. (101,184) .. controls (101,197.81) and (83.09,209) .. (61,209) .. controls (38.91,209) and (21,197.81) .. (21,184) -- cycle ;
\draw    (140,59) -- (141,184) ;
\draw    (101,59) -- (140,59) ;
\draw    (101,121) -- (140,121) ;
\draw    (140,121) -- (176,121) ;
\draw [shift={(178,121)}, rotate = 180] [color={rgb, 255:red, 0; green, 0; blue, 0 }  ][line width=0.75]    (10.93,-3.29) .. controls (6.95,-1.4) and (3.31,-0.3) .. (0,0) .. controls (3.31,0.3) and (6.95,1.4) .. (10.93,3.29)   ;
\draw    (101,184) -- (141,184) ;
\draw (169,67) node [anchor=north west][inner sep=0.75pt]   [align=left] {\textbf{Trained recursive DeepONet}};
\draw (202,108) node [anchor=north west][inner sep=0.75pt]    {$\mathcal{F}_{\theta ^{*}}\left(\check{x}( t_{n}) ,u_{m}^{n}\right)( h_{n})$};
\draw (392,109) node [anchor=north west][inner sep=0.75pt]    {$\check{x}(t_{n} +h_{n})$};
\draw (42,48) node [anchor=north west][inner sep=0.75pt]    {$\check{x}(t_{n})$};
\draw (48,110) node [anchor=north west][inner sep=0.75pt]    {$u_{m}^{n}$};
\draw (51,175) node [anchor=north west][inner sep=0.75pt]    {$h_{n}$};
\draw (38,219) node [anchor=north west][inner sep=0.75pt]   [align=left] {Inputs};
\draw (381,153) node [anchor=north west][inner sep=0.75pt]   [align=left] {Predicted state};
\end{tikzpicture}
\caption{{The trained DeepONet $\mathcal{F}_{\theta^*}$ recursively predicts the next state of a non-autonomous system $\check{x}(t_n + h_n)$ given the current state $\check{x}(t_n)$, the local approximation of the input $u_m^n$, and the step-size $h_n$.}}
\label{fig:recursive-DeepONet}
\end{figure}
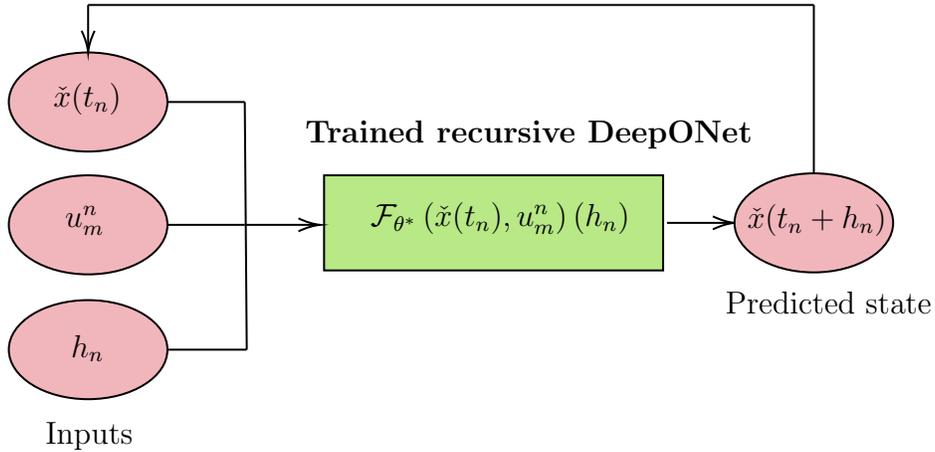
\subsection{Error Bound for the DeepONet-based Numerical  Scheme} \label{ssec:error-bound}
\textit{Assumptions.} We let the input \textit{function}~$u \in V \subset C[a,b]$ where $V$ is compact. We also assume the \textit{unknown} vector field~$f:\mathcal{X} \times \mathcal{U} \to \mathcal{X}$ is Lipschitz in $x$ and $u$, \ie 
\begin{align*}
    & \|f(x_1, u) - f(x_2, u)\|\leq C_1\|x_1 - x_2\|,\\
    &  \|f(x, u_1) - f(x, u_2)\|\leq C_1\|u_1 - u_2\|,
\end{align*}
where $C_1 > 0$ is a constant, and $x_1, x_2, u_1, u_2$ are in some proper space. Such an assumption is common in engineering, as $f$ is often differentiable with respect to $x$ and $u$.

The following lemma, presented in~\cite{qin2021data}, provides us with an alternative form of the local solution operator~\eqref{eq:approximate-solution-operator}. Such a form will be used later when estimating the error bound.
\begin{lemma} \label{lemma:functional-form}
Consider the local solution operator~$\mathcal{F}(\tilde{x}(t_n), u_m^n)(h_n)$. Then, there exists a function $\Phi: \mathbb{R}^n \times \mathbb{R}^{n_s} \times \mathbb{R} \to \mathbb{R}^n$, which depends on $f$, such that
\begin{align}
    \tilde{x}(t_{n+1}) = \mathcal{F}(\tilde{x}(t_n), u_m^n)(h_n) = \Phi(\tilde{x}(t_n), u_m^n, h_n),
    \label{eq_tilde}
\end{align}
for any $t_n \in \mathcal{P}$. In the above, $u^n_m \in \mathbb{R}^{n_s}$ locally characterizes~$u_m$ on the mesh $t_n=d_0^n < d_1^n < \ldots < d_{n_s-1}^n = t_{n+1}$.
\end{lemma}
We now provide the error estimation of the proposed DeepONet prediction scheme detailed in Algorithm~\ref{alg:prediction-scheme}. For the ease of notation, we denote $x_n = x(t_n)$ and $\check{x}_n = \check{x}(t_n)$ for all $t_n \in \mathcal{P}$.
\begin{lemma}
For any $t_{n}\in\mathcal{P}$, we have
\begin{align}
    \|x_n - \tilde{x}_n\|\leq \frac{1- \Bar{C}^n}{1-\Bar{C}}\Bar{e}(u_m):= \Tilde{C}\Bar{e}_m,
    \label{tilde_error}
\end{align}
where $\Bar{C} = e^{C_1 h}$, $\Bar{e}(u_m) = \max_{n}\Bar{e}_n(u_m)$, and  $\Bar{e}_n(u_m) = C_1 h_n \kappa_n(m) e^{C_1 h_n}$.
\end{lemma}
\begin{proof}
Let~$a = t_n$ and $s\in[t_n, t_{n+1}]$ in the solution operators~\eqref{eq:solution-operator} and~\eqref{eq:approximate-solution-operator}. Then, subtracting~\eqref{eq:approximate-solution-operator} from~\eqref{eq:solution-operator} gives 
\begin{align*}
    \|x(s) - \tilde{x}(s)\|& \leq \|x_n-\tilde{x}_n\| + \int_{t_n}^{s}\|f(x(t), u(t))-f(\tilde{x}(t), u_m)\| dt\\
    &\leq \|x_n-\tilde{x}_n\|  + C_1\int_{t_n}^{s}\|u(t)-u_m\| dt  +C_1\int_{t_n}^{s}\|x(t)-\tilde{x}(t)\|dt \\
    &\leq \|x_n-\tilde{x}_n\| + C_1 h_n \kappa_n(m) + C_1\int_{t_n}^{s}\|x(t)-\tilde{x}(t)\| dt,
\end{align*}
where $\kappa_n(m)$ is the local approximation error of the input within the interval $[t_n, t_{n+1}]$:
\begin{align*}
    \max_{s\in[t_n, t_{n+1}]}|u(s) - u_m|\leq \kappa_n(m),  
\end{align*}
such that
$$\kappa_n(m)\rightarrow 0 \text{ when the number of sensors}~m \to \infty.$$
We refer the interested reader to equation~(4) of~\cite{lu2019deeponet} for details of such an approximation for the input. Set $s = t_{n+1}$ and apply Gronwall's inequality, we then have,
\begin{align}
    \|x_{n+1} - \tilde{x}_{n+1}\|& \leq \|x_n-\tilde{x}_n\|e^{C_1 h_n} + \underbrace{C_1 h_n \kappa_n(m) e^{C_1 h_n}}_{\Bar{e}_{n}(u_m)}.
    \label{tilde_gronwall}
\end{align}
Taking~$\Bar{e}(u_m) = \max_n \Bar{e}_n(u_m)$ gives
\begin{align*}
    \| x_{n+1} - \tilde{x}_{n+1}\| \leq \| x_n - \tilde{x}_n \| e^{C_1 h_n} + \Bar{e}(u_m).
\end{align*}
The bound then follows immediately due to $ x(t_0) = \tilde{x}_0$.
\end{proof}

Before we estimate the cumulative error of the DeepONet-assisted solution $\check{x}_n$, we review the universal approximation theorem of neural networks for high-dimensional functions~\cite{cybenko1989approximation}. To this end, given $h_n$, we define the following vector-valued continuous function~$\varphi:\mathbb{R}^n \times \mathbb{R}^{n_s} \to \mathbb{R}^{n}$
\begin{align*}
    \varphi(y_n, u_m^n) = \mathcal{F}(y_n, u_m^n)(h_n) = \Phi(y_n, u_m^n,h_n),
\end{align*}
where $y_n\in\mathbb{R}^n$. Then, by the universal approximation theorem, for $\Bar{e}(u_m)>0$, there exist $W_1\in\mathbb{R}^{K \times (n+n_s)}$, $b_1\in\mathbb{R}^{K}$, $W_2\in\mathbb{R}^{n \times K}$ and $b_2\in\mathbb{R}^{n}$ such that
\begin{align}
    \bigg \|\varphi(y_n,u_m^n) - \left(W_2 \sigma (W_1 [y_n,u_m^n]^\top + b_1)+b_2\right) \bigg\| <\Bar{e}(u_m).
    \label{universal_approximation}
\end{align}
Here, the two-layer network represents the DeepONet for a given~$h_n$, \ie 
$$\left(W_2 \sigma (W_1 [y_n,u_m^n]^\top + b_1)+b_2\right) \equiv \mathcal{F}_{\theta^*}(y_n, u_m^n).$$

The following lemma estimates the cumulative error between the DeepONet-assisted solution $\check{x}$ (obtained via Algorithm \ref{alg:prediction-scheme}) and the solution $\tilde{x}$ of the approximate system that satisfies \eqref{eq_tilde}.
\begin{lemma}
Assume $\Phi$ is Lipschitz in $x$ with Lipschitz constant $C_2$. Suppose the DeepONet is well trained so that the network satisfies \eqref{universal_approximation}. Then, we have the following estimate:
\begin{align}
    \|\check{x}_{n} - \tilde{x}_n\| \leq \check{C} \Bar{e}(u_m),
    \label{check_error}
\end{align}
where $\check{C} = \frac{1 - C_2^n}{1-C_2}$.
\end{lemma}
\begin{proof}
It follows from the universal approximation theorem of neural networks \eqref{universal_approximation} and $\Phi$ being Lipschitz that, 
\begin{align*}
    \|\check{x}_{n+1} - \tilde{x}_{n+1} )\| & = \|\mathcal{F}_{\theta^*} (\check{x}_n,u_m^n) - \Phi(\tilde{x}_n, u_m^n, h_n) \|\\
    & = \|\mathcal{F}_{\theta^*} (\check{x}_n,u_m^n) -  \varphi(\check{x}_n, u_m^n)  \| \\ &\qquad + \| \Phi(\check{x}_n,u_m^n,h_n ) -  \Phi(\tilde{x}_n,u_m^n,h_n )\|\\
    & \leq \Bar{e}(u_m)  + C_2\| \check{x}_{n} - \tilde{x}_{n}  \|.
\end{align*}
The result follows immediately from $\tilde{x}_0 = \check{x}_0$.
\end{proof}
The following theorem summarizes the error of the proposed DeepONet scheme.
\begin{theorem}
For any $t_n\in\mathcal{P}$, we have
\begin{align}
    \|x_n - \check{x}_n \|\leq \tilde{C}\Bar{e}(u_m) + \check{C}\Bar{e}(u_m),
\end{align}
where $\Tilde{C}$, $\check{C}$ and $\Bar{e}(u_m)$ are, respectively,  the constants defined in \eqref{tilde_error}, \eqref{check_error}, and \eqref{tilde_error}.
\end{theorem}
We conclude this section by observing that the error bound found in this section is tighter than the error found in~\cite{qin2021data}, which behaves like $te^{ct}$ where $c$ is a positive constant.
\section{Data-Driven Runge-Kutta DeepONet Prediction Scheme}
\label{sec:data-driven-rk}
In this section, we propose a \textit{data-driven} Runge-Kutta explicit DeepONet-based scheme~\cite{iserles2009first} that predicts the new state vector~$\hat{x}(t_n + h_n)$ using the current state value~$\hat{x}(t_n)$, \ie
$$\hat{x}(t_n + h_n) = \hat{x}(t_n) + \frac{h_n}{2} (k_1 + k_2) .$$
Here $k_1$ and $k_2$ are, respectively, the estimates of~$f$ at $t_n$ and $t_{n+1}$. We compute these estimates (see equation~\eqref{eq:estimates}) using (i) the forward pass of trained DeepONet~$\mathcal{F}_{\theta^*}$ and (ii) automatic differentiation. Note that in~\eqref{eq:future-estimate}, we use the notation~$\bar{x}(t_{n+1})$ for the estimate of the state at $t_{n+1}$ obtained using the DeepONet's~$\mathcal{F}_{\theta^*}$ forward pass. 

We detail the proposed \textit{data-driven} RK explicit DeepONet-based scheme in~Algorithm~\ref{alg:data-driven-rk-scheme} and {Figure~\ref{fig:data-driven-RK-DeepONet}.} Two remarks about our algorithm are provided next. (i) For simplicity, we only present our scheme for the improved Euler method or RK-2~\cite{iserles2009first}. However, we remark that we can extend our idea to any explicit Runge-Kutta scheme. (ii) If $u_m^{n+1}$ is available at $t_n$, then we can compute $k_2$ as follows:
\begin{align} \label{eq:alternate-estimate}
k_2 = \frac{d}{dt}(\mathcal{F}_{\theta^*}(\check{x}(t_n), u^{n+1}_m)(h_n)).  \end{align}
Then, equations~\eqref{eq:current-estimate} and~\eqref{eq:alternate-estimate} will work as a predictor-corrector scheme with the updated input information. Other strategies can also be adopted within the proposed RK scheme. However, we let the design of such strategies for our future work. 
\begin{algorithm}[t]
\DontPrintSemicolon
\SetAlgoLined
\textbf{Require:} initial state vector~$x(a) = x_0$, partition~$\mathcal{P}$, input~$u^n_m$, for $n=1,\ldots,M-1$, and trained DeepONet~$\mathcal{F}_{\theta^*}$.\;
initialize $\check{x}(t_0) = x_0$\;
\For{$n = 0,\ldots,M-1$}{
  update the independent variable~$t_{n+1} = t_n + h_n$\;
  use the DeepONet's forward pass to compute
  $$\bar{x}(t_{n+1}) = \mathcal{F}_{\theta^*}(\check{x}(t_n), u^n_m)(h_n).$$}
  \vspace{-1.0em}\;
  use \textit{automatic differentiation} to estimate the vector field~$f$ at $t = t_n$ and $t = t_{n+1}$
  \begin{subequations} \label{eq:estimates}
  \begin{align}
  k_1 & =\frac{d}{dt}(\mathcal{F}_{\theta^*}(\check{x}(t_n), u^n_m)(0)) = f(\hat{x}(t_n),u_m^n), \label{eq:current-estimate} \\
  k_2 &= \frac{d}{dt}\bar{x}(t_{n+1}) = f(\bar{x}(t_{n+1}),u_m^n). \label{eq:future-estimate}
  \end{align}
  \end{subequations}
  \vspace{-1.0em}\;
  update the state vector with the improved Euler~(RK-2) step
  \begin{align} \label{eq:data-driven-RK}
  \check{x}(t_{n+1}) = \check{x}(t_n) + \frac{h_n}{2} (k_1 + k_2).
  \end{align}
  \vspace{-1.5em}\;
  \textbf{Return:} predicted response~$\{\check{x}(t_n) : t_n \in \mathcal{P}\}$.\;
 \caption{Data-Driven Runge-Kutta Scheme}
 \label{alg:data-driven-rk-scheme}
\end{algorithm}

\begin{figure}[t!]
\tikzset{every picture/.style={line width=0.75pt}}
\begin{tikzpicture}[x=0.75pt,y=0.65pt,yscale=-0.8,xscale=0.78]

\draw  [fill={rgb, 255:red, 184; green, 233; blue, 134}  ,fill opacity=1 ] (142,97) -- (374,97) -- (374,145) -- (142,145) -- cycle ;
\draw (651,80) -- (651,100) ;
\draw [shift={(651,102)}, rotate = 270] [color={rgb, 255:red, 0; green, 0; blue, 0 }  ][line width=0.75]    (10.93,-3.29) .. controls (6.95,-1.4) and (3.31,-0.3) .. (0,0) .. controls (3.31,0.3) and (6.95,1.4) .. (10.93,3.29)   ;
\draw  [fill={rgb, 255:red, 241; green, 182; blue, 188 }  ,fill opacity=1 ] (21,59) .. controls (21,45.19) and (38.91,34) .. (61,34) .. controls (83.09,34) and (101,45.19) .. (101,59) .. controls (101,72.81) and (83.09,84) .. (61,84) .. controls (38.91,84) and (21,72.81) .. (21,59) -- cycle ;
\draw    (61,10) -- (61,32) ;
\draw [shift={(61,34)}, rotate = 270] [color={rgb, 255:red, 0; green, 0; blue, 0 }  ][line width=0.75]    (10.93,-3.29) .. controls (6.95,-1.4) and (3.31,-0.3) .. (0,0) .. controls (3.31,0.3) and (6.95,1.4) .. (10.93,3.29)   ;
\draw    (61,10) -- (751,11) ;
\draw  [fill={rgb, 255:red, 241; green, 182; blue, 188 }  ,fill opacity=1 ] (21,121) .. controls (21,107.19) and (38.91,96) .. (61,96) .. controls (83.09,96) and (101,107.19) .. (101,121) .. controls (101,134.81) and (83.09,146) .. (61,146) .. controls (38.91,146) and (21,134.81) .. (21,121) -- cycle ;
\draw  [fill={rgb, 255:red, 241; green, 182; blue, 188 }  ,fill opacity=1 ] (21,184) .. controls (21,170.19) and (38.91,159) .. (61,159) .. controls (83.09,159) and (101,170.19) .. (101,184) .. controls (101,197.81) and (83.09,209) .. (61,209) .. controls (38.91,209) and (21,197.81) .. (21,184) -- cycle ;
\draw    (120.5,59) -- (121.5,184) ;
\draw    (101,59) -- (120.5,59) ;
\draw    (101.5,121.5) -- (140.5,121.5) ;
\draw    (140,121) ;
\draw [shift={(140,121)}, rotate = 180] [color={rgb, 255:red, 0; green, 0; blue, 0 }  ][line width=0.75]    (10.93,-3.29) .. controls (6.95,-1.4) and (3.31,-0.3) .. (0,0) .. controls (3.31,0.3) and (6.95,1.4) .. (10.93,3.29)   ;
\draw    (101,184) -- (121.5,184) ;
\draw  [fill={rgb, 255:red, 184; green, 233; blue, 134 }  ,fill opacity=1 ] (404,65) -- (616,65) -- (616,113) -- (404,113) -- cycle ;
\draw  [fill={rgb, 255:red, 184; green, 233; blue, 134 }  ,fill opacity=1 ] (402,127) -- (612,127) -- (612,175) -- (402,175) -- cycle ;
\draw  [fill={rgb, 255:red, 184; green, 233; blue, 134 }  ,fill opacity=1 ] (636,100) -- (863,100) -- (863,148) -- (636,148) -- cycle ;
\draw    (386,89) -- (386,155) ;
\draw    (371,122) -- (386,122) ;
\draw    (386,89) -- (399,89) ;
\draw [shift={(401,89)}, rotate = 180] [color={rgb, 255:red, 0; green, 0; blue, 0 }  ][line width=0.75]    (10.93,-3.29) .. controls (6.95,-1.4) and (3.31,-0.3) .. (0,0) .. controls (3.31,0.3) and (6.95,1.4) .. (10.93,3.29)   ;
\draw    (386,155) -- (399,155) ;
\draw [shift={(401,155)}, rotate = 180] [color={rgb, 255:red, 0; green, 0; blue, 0 }  ][line width=0.75]    (10.93,-3.29) .. controls (6.95,-1.4) and (3.31,-0.3) .. (0,0) .. controls (3.31,0.3) and (6.95,1.4) .. (10.93,3.29)   ;
\draw    (651,80) -- (614,80) ;
\draw    (650,165) -- (615,165) ;
\draw    (650,165) -- (650,152) ;
\draw [shift={(650,150)}, rotate = 90] [color={rgb, 255:red, 0; green, 0; blue, 0 }  ][line width=0.75]    (10.93,-3.29) .. controls (6.95,-1.4) and (3.31,-0.3) .. (0,0) .. controls (3.31,0.3) and (6.95,1.4) .. (10.93,3.29)   ;
\draw    (751,11) -- (751,99) ;

\draw (128,69) node [anchor=north west][inner sep=0.75pt]   [align=left] {\begin{small}\textbf{Data-driven RK DeepONet}\end{small}};
\draw (147,106) node [anchor=north west][inner sep=0.75pt]    {\begin{footnotesize}$\overline{x}( t_{n+1})=\mathcal{F}_{\theta ^{*}}\left(\check{x}( t_{n}) ,u_{m}^{n}\right)( h_{n})$\end{footnotesize}};
\draw (41,46.4) node [anchor=north west][inner sep=0.75pt]    {$\check{x}(t_{n})$};
\draw (48,108.4) node [anchor=north west][inner sep=0.75pt]    {$u_{m}^{n}$};
\draw (51,173.4) node [anchor=north west][inner sep=0.75pt]    {$h_{n}$};
\draw (39,219) node [anchor=north west][inner sep=0.75pt]   [align=left] {Inputs};
\draw (699,160) node [anchor=north west][inner sep=0.75pt]   [align=left] {Predicted state};
\draw (408,70) node [anchor=north west][inner sep=0.75pt]{\begin{footnotesize}$k_{1}=\frac{d}{dt}\left(\mathcal{F}_{\theta ^{*}}\left(\check{x}( t_{n}) ,u_{m}^{n}\right)(0)\right)$\end{footnotesize}};
\draw (449,132) node [anchor=north west][inner sep=0.75pt]{\begin{footnotesize}$k_{2}=\frac{d}{dt}\overline{x}( t_{n+1})$\end{footnotesize}};
\draw (638,104) node [anchor=north west][inner sep=0.5pt]{\begin{footnotesize}$\check{x}(t_{n+1}) =\check{x}( t_{n})+\frac{h_{n}}{2}( k_{1} +k_{2})$\end{footnotesize}};
\draw (404,187) node [anchor=north west][inner sep=0.75pt]   [align=left] {Automatic differentiation};
\end{tikzpicture}
\caption{{The data-driven RK DeepONet takes the current state $\check{x}(t_n)$, the local approximation of the input $u_m^n$, and the step-size $h_n$ as inputs. It then recursively predicts the next state of a non-autonomous system using a data-driven RK-2 method, which employs the forward pass of the DeepONet and automatic differentiation to obtain the estimates of the vector field $k_1$ and $k_2$.}}
\label{fig:data-driven-RK-DeepONet}
\end{figure}
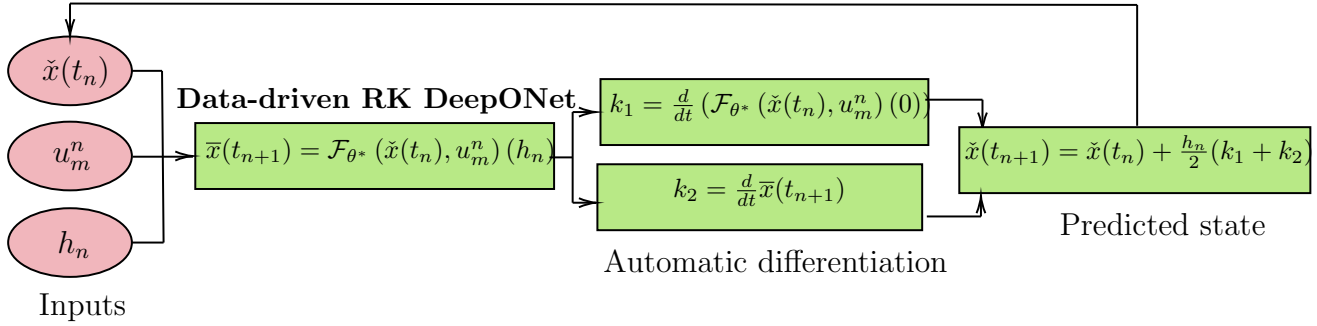

\subsection{Error Bound for the Data-Driven Runge-Kutta Scheme} \label{ssec:error-bound-rk}
Here we derive an improved conditional error bound estimate for~$\hat{x}(t_n)$. To that end, we start by rephrasing the universal approximation theorem of neural network for high-dimensional functions~\cite{cybenko1989approximation}, which we introduced in Section~\ref{ssec:error-bound}. For~$\Bar{e}(u_m)- C_4 h_n^2>0$, there exist $W_1\in\mathbb{R}^{K \times (n+n_s)}$, $b_1\in\mathbb{R}^{K}$, $W_2\in\mathbb{R}^{n \times K}$ and $b_2\in\mathbb{R}^{n}$ such that
\begin{align}
    \bigg \|\varphi(y_n,u_m^n) - \left(W_2 \sigma (W_1 [y_n,u_m^n]^\top + b_1)+b_2\right) \bigg\| < \epsilon,
    \label{universal_approximation_2}
\end{align}
where $\epsilon := \Bar{e}(u_m) - C_4 h^2$ and $C_4>0$ is constant. As before, the two-layer network represents the DeepONet~$\mathcal{F}_{\theta^*}$ for a given~$h_n$.

The following lemma estimates the error between~$\hat{x}$, predicted using the RK scheme (Algorithm~\ref{alg:data-driven-rk-scheme}), and $\tilde{x}$, obtained using the solution operator~\eqref{eq:approximate-solution-operator} of the approximate system~\eqref{eq:approximate-system}.
\begin{lemma} \label{lemma:rk-error-vs-approximate}
Assume $\Phi$ is Lipschitz in $x$ with Lipschitz constant~$C_2$. Suppose the DeepONet is well trained so that \eqref{universal_approximation_2} holds. Then, we have the estimate
\begin{align}
    \|\hat{x}(t_n) - \tilde{x}(t_n)\| \leq \hat{C} \Bar{e}(u_m),
    \label{hat_error}
\end{align}
where $\hat{C} =  \frac{C_1 h}{2}\frac{1 - C_3^n}{1-C_3}$, and $ C_3 = \left(1+(1+C_2)\frac{C_1 h}{2}\right)$.

\end{lemma}

\begin{proof}
We denote $\tilde{x}_{n} = \tilde{x}(t_n)$, $\Bar{x}_n = \Bar{x}(t_n)$, and $\hat{x}_n = \hat{x}(t_n)$. Then, it follows from~$f$ being Lipschitz that,
\begin{align}
    \| \tilde{x}_{n+1} - \hat{x}_{n+1} \| & =
    \| \tilde{x}_n + \int_{t_n}^{t_{n+1}}f(\tilde{x}(s), u_m^n)ds
    - \hat{x}_{n} \nonumber\\ &\qquad - \frac{h_n}{2} (f(\bar{x}_{n+1}, u_m^n) + f(\hat{x}_n, u_m^n))\| \nonumber\\
    & \leq \|\tilde{x}_n - \hat{x}_{n}\| + \frac{h_n}{2}\|f(\tilde{x}_n, u_m^n)-f(\hat{x}_n, u_m^n) \nonumber \\ &\qquad + f(\tilde{x}_{n+1}, u_m^{n})
    - f(\bar{x}_{n+1}, u_m^{n})\|  +\textcolor{black}{\mathcal{O}(h^3)}\nonumber\\
    & \leq \|\tilde{x}_n - \hat{x}_{n}\| + \frac{C_1h_n}{2}\|\tilde{x}_n-\hat{x}_n\| \nonumber\\ &\qquad+\frac{C_1 h_n}{2}\|\tilde{x}_{n+1}-\bar{x}_{n+1}\|
   \textcolor{black}{ +\mathcal{O}(h^3)} \nonumber\\
    & \leq \left(1 + \frac{C_1 h_n}{2} \right)\|\tilde{x}_n-\hat{x}_n\| \nonumber \\&\qquad +
    \frac{C_1 h_n}{2}\|\tilde{x}_{n+1}-\bar{x}_{n+1}\|
    +\mathcal{O}(h^3).
    \label{eqn1}
\end{align}
An estimate of the above term~$\|\tilde{x}_{n+1}-\bar{x}_{n+1}\|$ is
\begin{align}
    \|\tilde{x}_{n+1}-\bar{x}_{n+1}\| & = \|\Phi(\tilde{x}_n, u_m^n, h_n) - \mathcal{F}_{\theta^*}(\hat{x}_n,u_m^n)\|\nonumber \\
    & \leq \|\Phi(\tilde{x}_n, u_m^n, h_n) - \Phi(\hat{x}_n, u_m^n, h_n) \nonumber \\&\qquad~\quad + \varphi(\hat{x}_n, u_m^n) - \mathcal{F}_{\theta^*}(\hat{x}_n,u_m^n) \| \nonumber \\
    &\leq C_2\|\tilde{x}_n-\hat{x}_n\| +\Bar{e}(u_m)\textcolor{black}{-C_4 h^2},
    \label{eqn2}
\end{align}
where $C_4>0$ is a constant. Substituting \eqref{eqn2} back into \eqref{eqn1} yields
\begin{align}
    \| \tilde{x}_{n+1} - \hat{x}_{n+1} \| & \leq \left(1+(1+C_2)\frac{C_1 h_n}{2}\right)\|\tilde{x}_n-\hat{x}_n\| \nonumber \\&\qquad~\qquad + \frac{C_1 h_n}{2}\Bar{e}(u_m).
\end{align}
Recursive estimation and $\hat{x}_0 = \tilde{x}_0$ gives the desired error bound~\eqref{hat_error}.
\end{proof}
Two remarks about the error bound are as follows. (i)~Note that the proposed \textit{data-driven} RK DeepONet-based scheme provides an improved error bound~\eqref{hat_error} when compared to the bound obtained in~\eqref{check_error}. More specifically, the growth factor~$C_3$ behaves like~$C_2$ in~\eqref{check_error}. However, when $h_n < \frac{1}{C_1}\cdot\frac{2C_2-2}{C_2+1}$, a smaller factor can be derived. (ii)~We can extend the proof provided here for the RK-2 scheme to any other RK explicit scheme. We will analyze this in our future work. Let us conclude this section with the following theorem that summarizes the error of the proposed~\textit{data-driven} RK scheme.
\begin{theorem}
For any $t_n\in\mathcal{P}$, we have
\begin{align}
    \|x(t_n) - \hat{x}(t_n)\|\leq \tilde{C}\Bar{e}(u_m) + \hat{C}\Bar{e}(u_m),
\end{align}
where $\Tilde{C}$, $\hat{C}$ and $\Bar{e}(u_m)$ are, respectively, the constants defined in~\eqref{tilde_error}, \eqref{hat_error}, and \eqref{tilde_error}.
\end{theorem}

\section{Numerical Experiments} \label{sec:numerical-experiments}
To evaluate our framework, we tested the proposed DeepONets on five tasks: the autonomous Lorentz 63 system (in Section~\ref{ssec:Lorentz}), the predator-prey dynamics with control (in Section~\ref{ssec:predator-prey}), the pendulum swing-up (in Section~\ref{ssec:pendulum}), the cart-pole system (in Section~\ref{ssec:cartpole}), and a power engineering application (in Section~\ref{ssec:power-eng}). For the three control tasks, we used only $n_s=1$ sensor. The reasons for selecting only one sensor are twofold. First, we want to show that DeepONet is effective even when the input signal is encoded with minimal information. For reference, in~\cite{qin2021data}, the authors encoded the input signals (used in their experiments) with at least $n_s = 3$ interpolation points (sensors). Second, the $n_s=1$ sensor scenario resembles the scenario of sampled-data control systems~\cite{wittenmark2002computer} or reinforcement learning tasks~\cite{sutton2018reinforcement} with continuous action space. Finally, for the power engineering application task, we used $n_s=2$ sensors. 

\textit{Training dataset.} For each of the three continuous control tasks (predator-prey, pendulum, and cart-pole systems), we generate the training dataset~$\mathcal{D}_\text{train}$ as follows. We use Runge-Kutta~(RK-4)~\cite{iserles2009first} to simulate $N_\text{train}$ trajectories of size \textit{two}. For each trajectory, the input to RK-4 is the initial condition~$x(t_n)$ uniformly sampled from~$\mathcal{X}$ and the input~$u(t_n)$ uniformly sampled from the set~$\mathcal{U}$. The output from the RK-4 algorithm is the state~$x(t_n + h)$, where $h$ is uniformly sampled from the interval~$[0,0.25]$. Such a procedure gives the dataset:
$$\mathcal{D}_\text{train}=\{(x_i(t_n),u_i(t_n)),h_i, x_i(t_n + h_i)\}_{i=1}^{N_\text{train}}.$$

\textit{Training protocol and neural networks.} We implemented our framework using JAX~\cite{jax2018github}\footnote{We will publish the code on GitHub after publication.} The neural networks for the Branch and Trunk Nets are the modified fully-connected networks proposed in~\cite{wang2021understanding} and used in our previous paper~\cite{moya2023deeponet}. We trained the parameters of the networks using Adam~\cite{kingma2014adam}. Moreover, we selected (i) the default hyper-parameters for the Adam algorithm and (ii) an initial learning rate of $\eta = 0.001$ that exponentially decays every 2000 epochs.
\subsection{The Autonomous Lorentz 63 System} \label{ssec:Lorentz}
To evaluate the proposed framework, we first consider the autonomous and chaotic Lorenz 63 system with the following dynamics:
\begin{gather}
\begin{aligned} \dot{x} &= \sigma(y-x), \\ \dot{y} &= x(\rho - z) - y, \\ \dot{z} &= xy - \beta z, 
\end{aligned} \label{eq:Lorentz}
\end{gather}
with parameters $\sigma=10,~\rho=28,$  and $\beta = 8/3.$ Notice that, compared to non-autonomous systems, autonomous systems only require the previous state to predict the next state, simplifying the learning problem. Nevertheless, we do recognize that the chaotic nature of the Lorentz system can pose a challenge to error accumulation over time. To address this, we have arbitrarily increased the size of the training dataset, which is described next.

The \textit{training dataset} for this autonomous system is provided as a collection of 20000 scatter one-step responses, or trajectories of size two. Specifically, $D_\text{train} = \{x_i(t_n), h_i, x_i(t_n + h_i) \}_{i=1}^{N_\text{train}}$, where $x_i(t_n)$ is sampled from the state space $\mathcal{X}:=[-17,20] \times [-23,28] \times [0,50]$, and the step sizes $h_i$ are sampled from the interval $[0,0.02]$. Note that we keep the step size small to track the error accumulation of the chaotic system, but this requires us to increase the size of our training dataset and normalize the state space.

The trained DeepONet is used to predict the response of the autonomous Lorentz 63 system over the time-domain $[0,20]$ seconds, for the initial condition $(x(0),y(0),z(0))=(0,1,1)$, using a uniform partition $\mathcal{P}$ with a fixed step size of $h=0.01$. Figure~\ref{fig:Lorentz} shows the comparison between the recursive DeepONet prediction obtained from Algorithm~\ref{alg:prediction-scheme} and the true response of the Lorentz system's state variables $x(t)$ and $y(t)$. The results demonstrate excellent agreement between the proposed method and the true values, despite the chaotic nature of the autonomous Lorentz system.

\begin{figure}[t!]
\centering
\begin{subfigure}[b]{0.49\textwidth}
\centering
\includegraphics[width=1.0\textwidth, height=6.0cm]{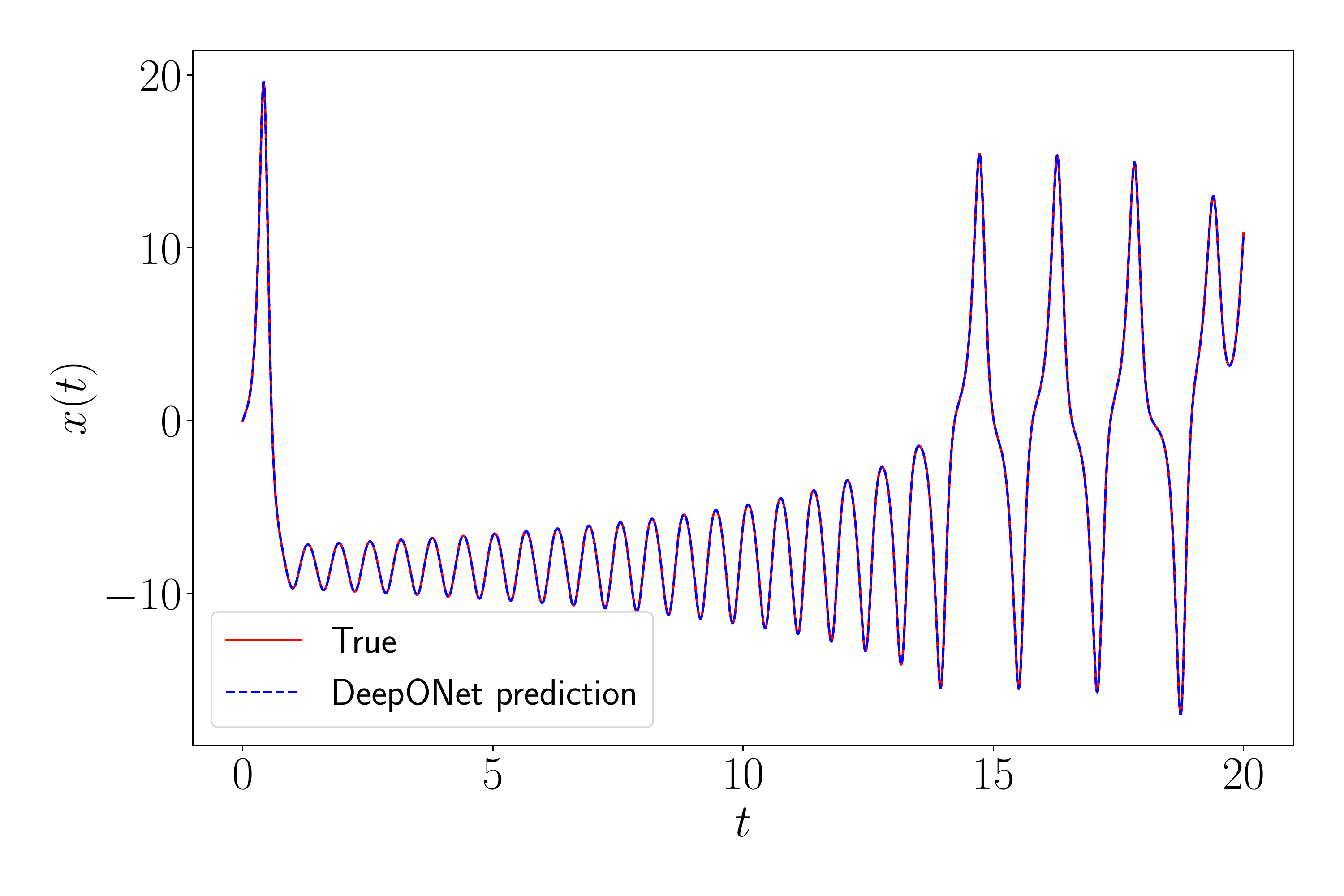}
\end{subfigure}
\begin{subfigure}[b]{0.49\textwidth}
\centering
\includegraphics[width=1.0\textwidth, height=6.0cm]{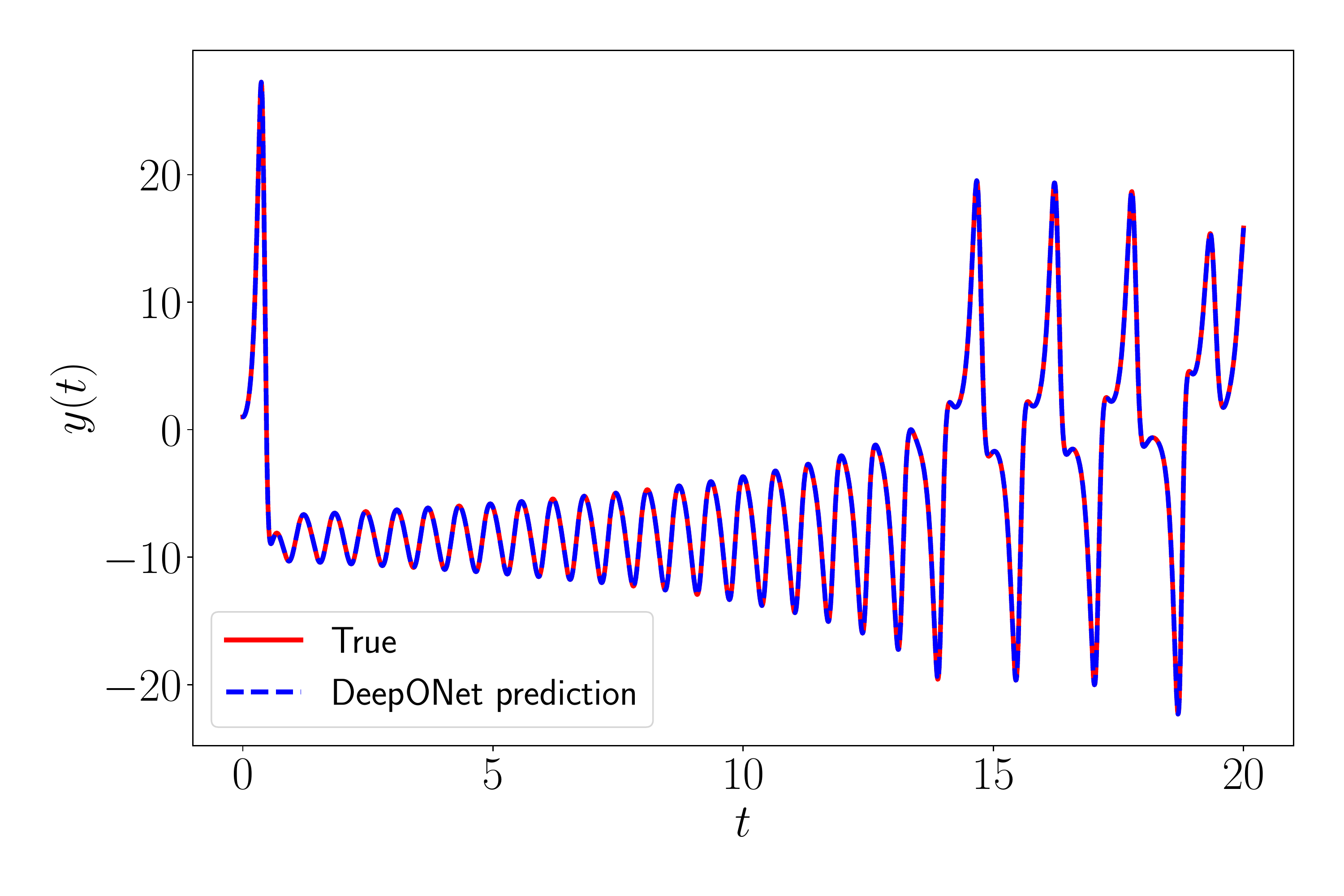}
\end{subfigure}
\caption{Comparison of DeepONet prediction with the actual trajectory of the autonomous Lorentz~\eqref{eq:Lorentz} system's state~$x = (x(t), y(t))^\top$ (\textit{left} is $x(t)$ and \textit{right} is $y(t)$) for the initial condition $(x(0),y(0),z(0))=(0,1,1)$ within the partition $\mathcal{P} \subset [0,20]$~(s) of constant step size $h = 0.01$.}
\label{fig:Lorentz}
\end{figure}

\subsection{The Predator-Prey Dynamics with Control} \label{ssec:predator-prey}
To evaluate our framework, we now consider the following Lotka-Volterra Predator-Prey system with input signal~$u(t)$:
\begin{gather}
    \begin{aligned}
         \dot{x}_1&= x_1 - x_1x_2 + u(t)\\
         \dot{x}_2&= -x_2+x_1x_2.
    \end{aligned}\label{eq:predator-prey}
\end{gather}
The system~\eqref{eq:predator-prey} was also studied in~\cite{qin2021data}, where the authors encoded~$u(t)$ using \textit{three} interpolation points. To train our DeepONet, we generated $N_\text{train}=2000$ trajectories with the initial condition~$x_i(t_n)$ (resp. input signal~$u_i(t_n)$) sampled from the state space~$\mathcal{X}:=[0,5]^2$ (resp. input space~$\mathcal{U}:=[0,5]$).

We use the trained DeepONet to predict the response of the predator-prey system (see equation \eqref{eq:predator-prey}) to the input signal $u(t) = \sin(t/3) + \cos(t) +2$ within a partition $\mathcal{P} \subset [0,100]$~(s) with a constant step size $h = 0.1 \equiv t_{n+1} - t_n$, where $t_n, t_{n+1} \in \mathcal{P}$. Figure~\ref{fig:predator-prey-prediction} compares DeepONet's long-term prediction with the true trajectory. Note that the predicted trajectory agrees very well with the true trajectory for both states, $x = (x_1,x_2)^\top$. The $L_2$-relative errors for $x_1$ and $x_2$ are summarized in Table \ref{table:pperror}.
\begin{table}[t]
\centering
\begin{tabular}{| c | c  c|}
\hline
& $x_1$ & $x_2$\\
\hline
\hline
$L_2$ relative error & 2.42 \% & 0.93 \%\\
\hline
\end{tabular}
\caption{Relative errors of predator and prey dynamics with control. }
\label{table:pperror}
\end{table}

\begin{figure}[t!]
\centering
\begin{subfigure}[b]{0.49\textwidth}
\centering
\includegraphics[width=1.0\textwidth, height=6.0cm]{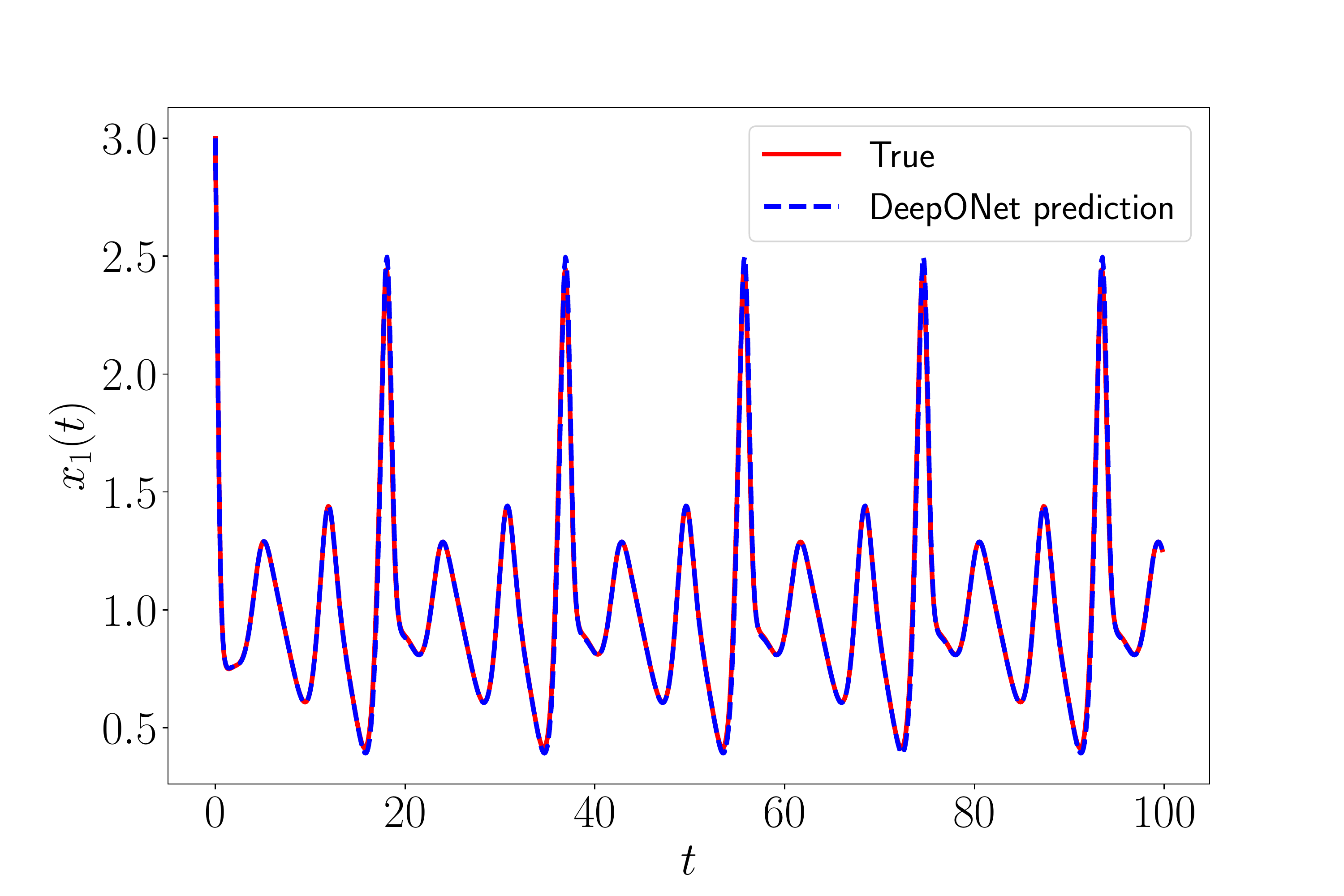}
\end{subfigure}
\begin{subfigure}[b]{0.49\textwidth}
\centering
\includegraphics[width=1.0\textwidth, height=6.0cm]{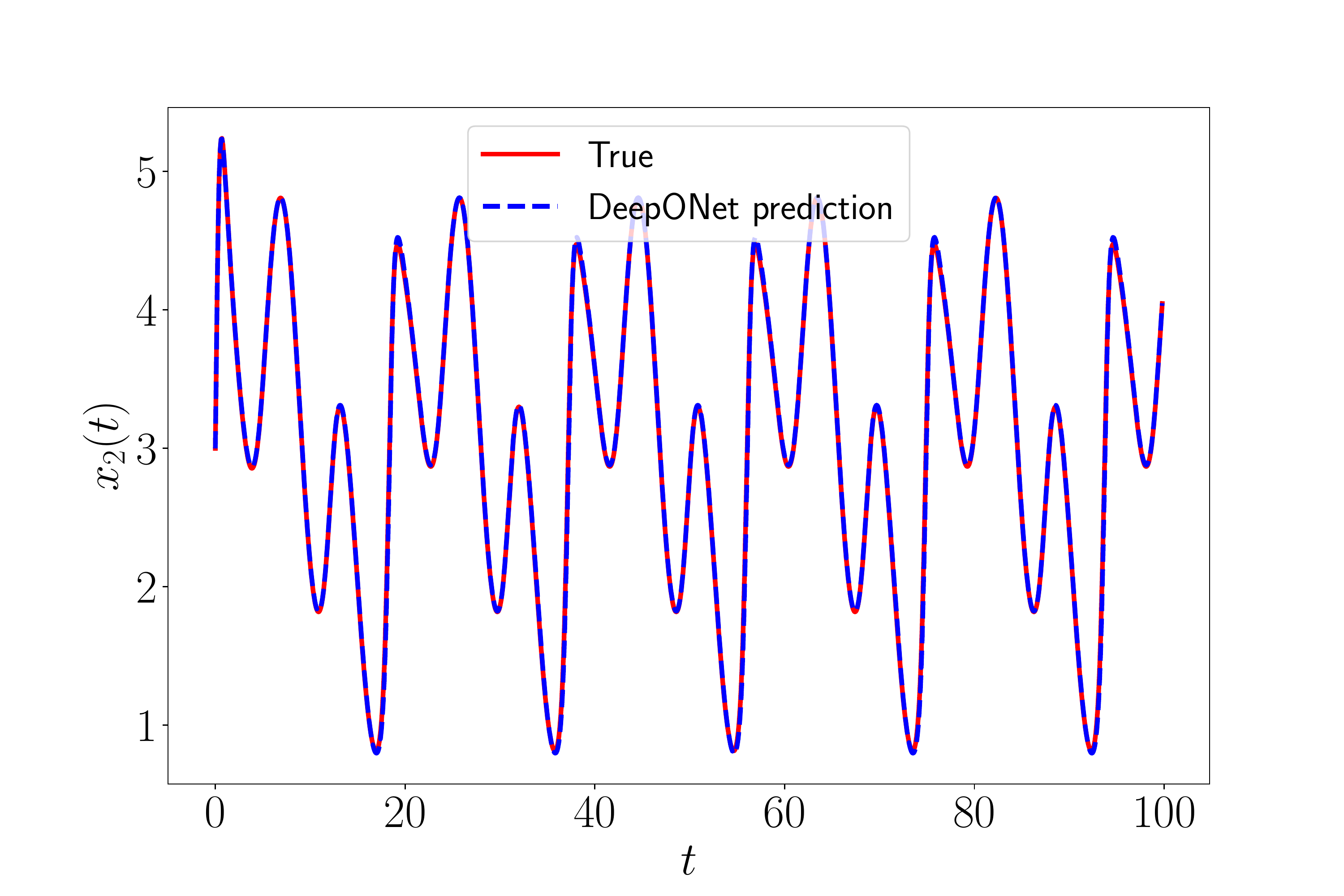}
\end{subfigure}
\caption{Comparison of DeepONet prediction with the actual trajectory of the predator-prey~\eqref{eq:predator-prey} system's state~$x = (x_1(t), x_2(t))^\top$ (\textit{left} is $x_1(t)$ and \textit{right} is $x_2(t)$) response to the input signal~$u(t) = \sin(t/3) + \cos(t) +2$ within the partition $\mathcal{P} \subset [0,100]$~(s) of constant step size $h = 0.1$.}
\label{fig:predator-prey-prediction}
\end{figure}

\subsection{Pendulum Swing-Up} \label{ssec:pendulum}
Let us now consider the following pendulum swing-up control system:
\begin{align} \label{eq:pendulum-swing-up}
    \ddot{\theta} \left(\frac{1}{4} m l^2 + I \right) + \frac{1}{2} m l g \sin \theta = u(t) - b \dot{\theta},
\end{align}
where $x = (\theta, \dot{\theta})^\top \in \mathcal{X}$ is the state vector, $\theta$ the pendulum's angle, $\dot{\theta}$ the angular velocity, and $u(t) \in \mathcal{U}$ the control torque. We set the parameters to the following values. The pendulum's mass is~$m = 1$~(kg), the length is~$l=1$~(m), the moment of inertia of the pendulum around the midpoint is~$I = \frac{1}{12}ml^2$, and the friction coefficient is~$b = 0.01$~(sNm/rad).

We trained the proposed framework using $N_\text{train}=5000$ samples, each consisting of a tuple $\{(x(t_n),u(t_n)), h_n, x(t_n + h_n)\}$. The initial condition, $x(t_n)$ (the state of the non-autonomous pendulum system), was sampled from the uniform distribution $\mathcal{X}:=[-\pi,\pi] \times [-8,8]$. The step size $h_n$ was obtained by uniformly sampling from the interval $(0, 0.02]$. The control torque, $u(t_n)$, was constant on the interval $[t_n, t_{n+1}]$ and was sampled from the uniform distribution $\mathcal{U}:=[-2,2]$. Finally, we obtained the terminal state, $x(t_n + h_n)$, by evolving according to the non-autonomous dynamics (see equation \eqref{eq:pendulum-swing-up}). It is worth noting that the the proposed operator learning setting allows us to handle datasets with incomplete data, as $h_n = t_{n+1}-t_n$ is not identical for all $5000$ training samples. Our methods thus have much more flexibility in terms of data preparation.

\textit{Stable response.} We begin by using the trained DeepONet to predict the pendulum's response to the input $u_1(t) = - 0.8 \dot{\theta}(t)$ within the partition $\mathcal{P} \subset [0,10]$ (s) with a step size of $h = 0.1$ (s). This input yields state trajectories $\{(\theta(t_n), \dot{\theta}(t_n)):t_n \in \mathcal{P}\}$ that settle to an asymptotic equilibrium point. Figure~\ref{fig:pendulum-prediction} shows the excellent agreement between the predicted and actual trajectories.
\begin{figure}[t!]
\centering
\begin{subfigure}[b]{0.49\textwidth}
\centering
\includegraphics[width=1.1\textwidth, height=6.0cm]{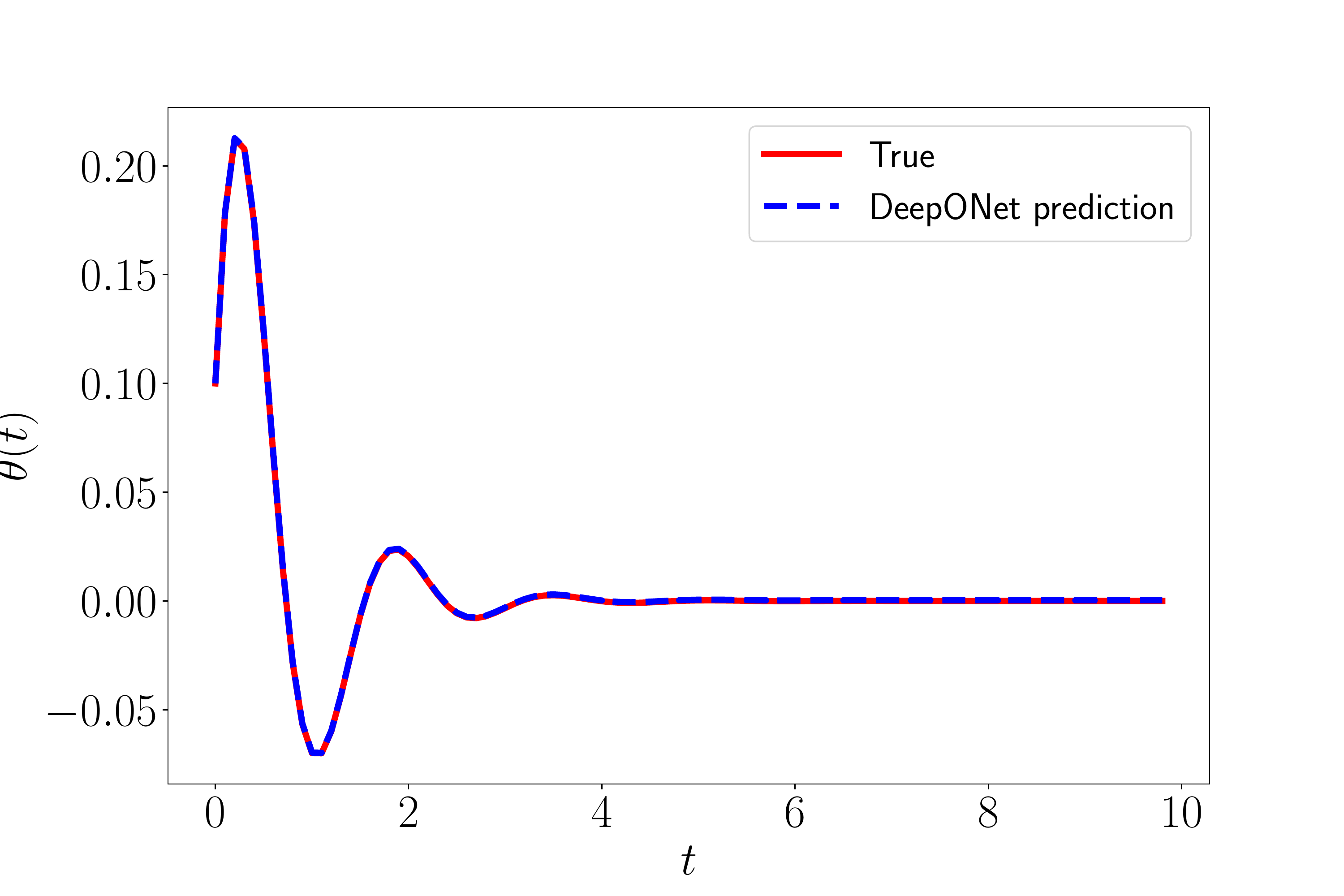}
\end{subfigure}
\begin{subfigure}[b]{0.48\textwidth}
\centering
\includegraphics[width=1.1\textwidth, height=6.0cm]{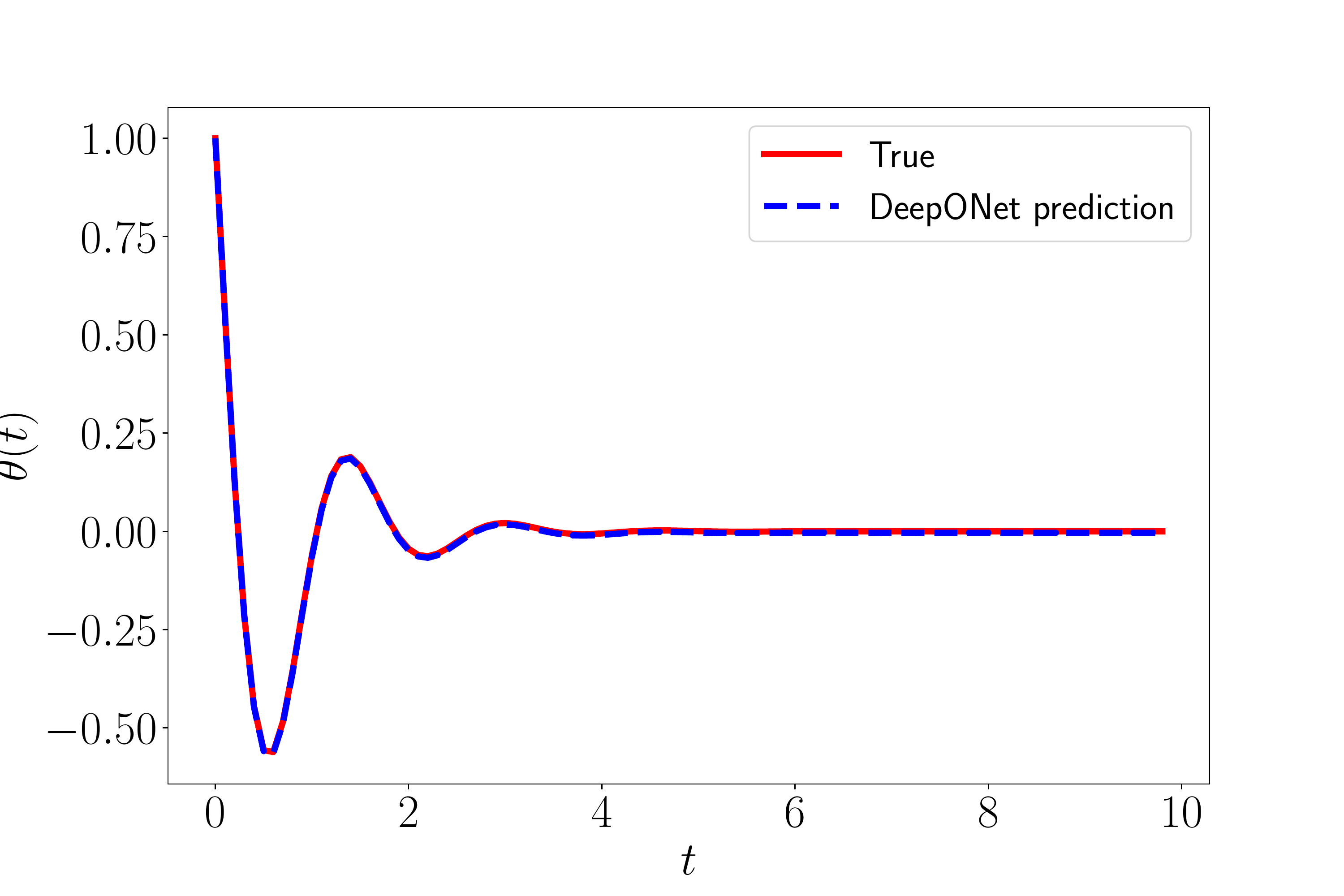}
\end{subfigure}
\caption{Comparison of DeepONet prediction with the actual trajectories of the pendulum~\eqref{eq:pendulum-swing-up} system's state~$x = (\theta(t), \dot{\theta}(t))^\top$ (\textit{left} is angle $\theta(t)$ and \textit{right} is velocity $\dot{\theta}(t)$) response to the input signal~$u(t) = -0.8 \dot{\theta}(t)$ within the partition $\mathcal{P} \subset [0,10]$~(s) of constant step size $h = 0.1$. }
\label{fig:pendulum-prediction}
\end{figure}

To test the predictive power of the proposed framework, we compute the average and standard deviation (st. dev.) of the $L_2$-relative error between the predicted and actual response of the pendulum to the following: (i) the control torque $u_1(t)$, and (ii) $100$ initial conditions sampled from the set $\mathcal{X}_o := \{\theta, \dot{\theta}:\theta \in [-\pi/2,\pi/2], \dot{\theta} = 0\}$. Table~\ref{table:pendulum-error} reports the results, demonstrating that DeepONet consistently maintains an average $L_2$-relative error of below $1.1\%$ and $3.4\%$ for the angle $\theta(t)$ and the angular velocity $\dot{\theta}(t)$ trajectories, respectively. Please refer to Table \ref{table:pendulum-error} for a summary.

\begin{table}[t]
\centering
\begin{tabular}{| c | c  c|}
\hline
& angle~$\theta(t)$ & angular velocity~$\dot{\theta}(t)$\\
\hline
\hline
mean $L_2$ & 1.056 \% & 3.356 \%\\
st.dev. $L_2$ & 2.509 \% & 8.234\% \\
\hline
\end{tabular}
\caption{The average and standard deviation (st.dev.) of the $L_2-$ relative error between the predicted and actual response trajectories of the pendulum system~\eqref{eq:pendulum-swing-up} to (i) the control torque~$u_1(t) = -0.8 \dot{\theta(t)}$ and (ii) $100$ initial conditions uniformly sampled from the set $\mathcal{X}_o := \{\theta, \dot{\theta}:\theta \in [-\pi/2,\pi/2], \dot{\theta} = 0\}$.}
\label{table:pendulum-error}
\end{table}

\textit{Oscillatory response.} We now test the trained DeepONet for the control torque $u_2(t) = \sin(t/2)$ within the partition $\mathcal{P} \subset [0,10]$ (s) with a constant step size of $h=0.1$. Figure~\ref{fig:pendulum-prediction-u2-h0p1} depicts the excellent agreement between the predicted and actual oscillatory trajectory.

\begin{figure}[t!]
\centering
\begin{subfigure}[b]{0.49\textwidth}
\centering
\includegraphics[width=1.1\textwidth, height=6.0cm]{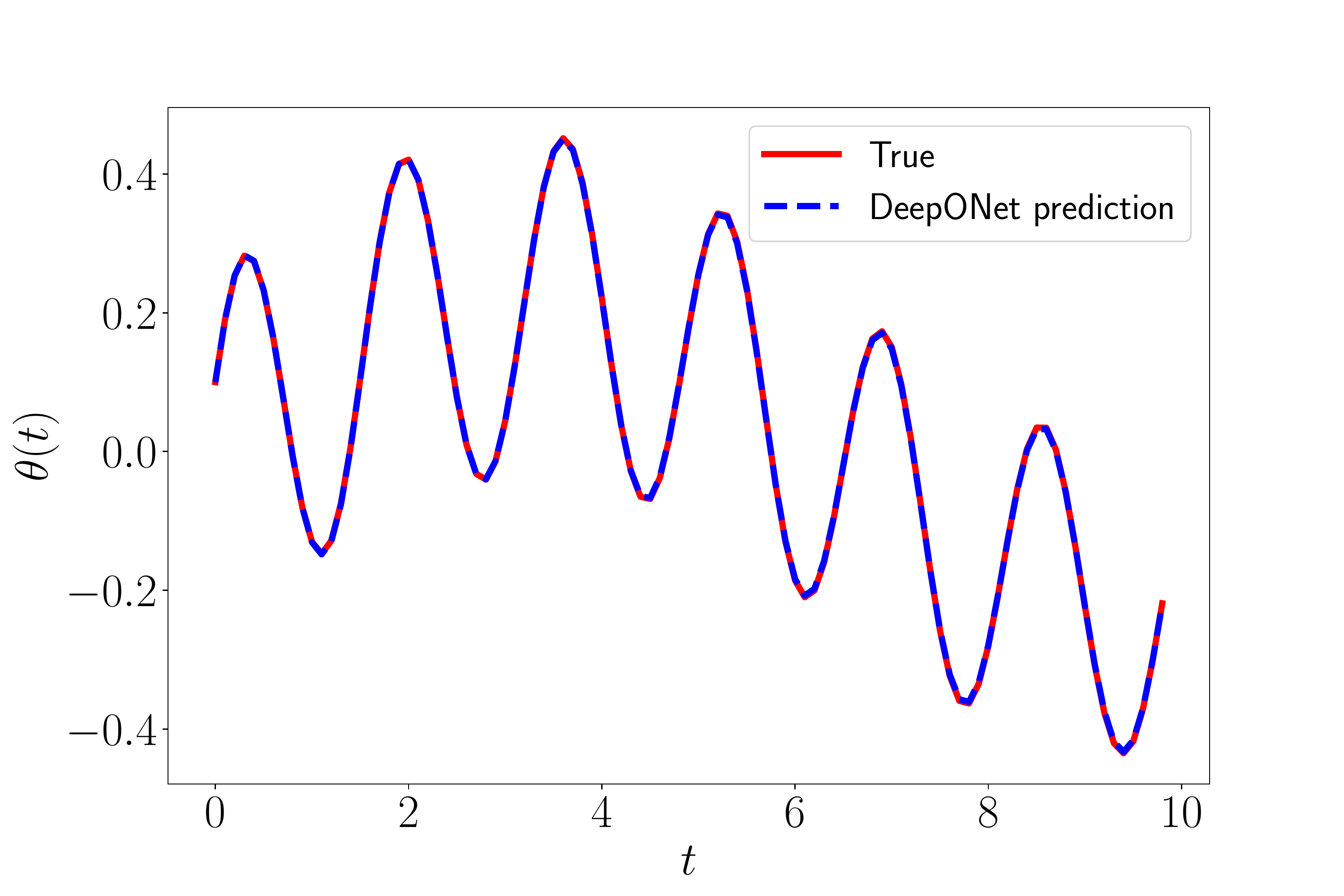}
\end{subfigure}
\begin{subfigure}[b]{0.49\textwidth}
\centering
\includegraphics[width=1.1\textwidth, height=6.0cm]{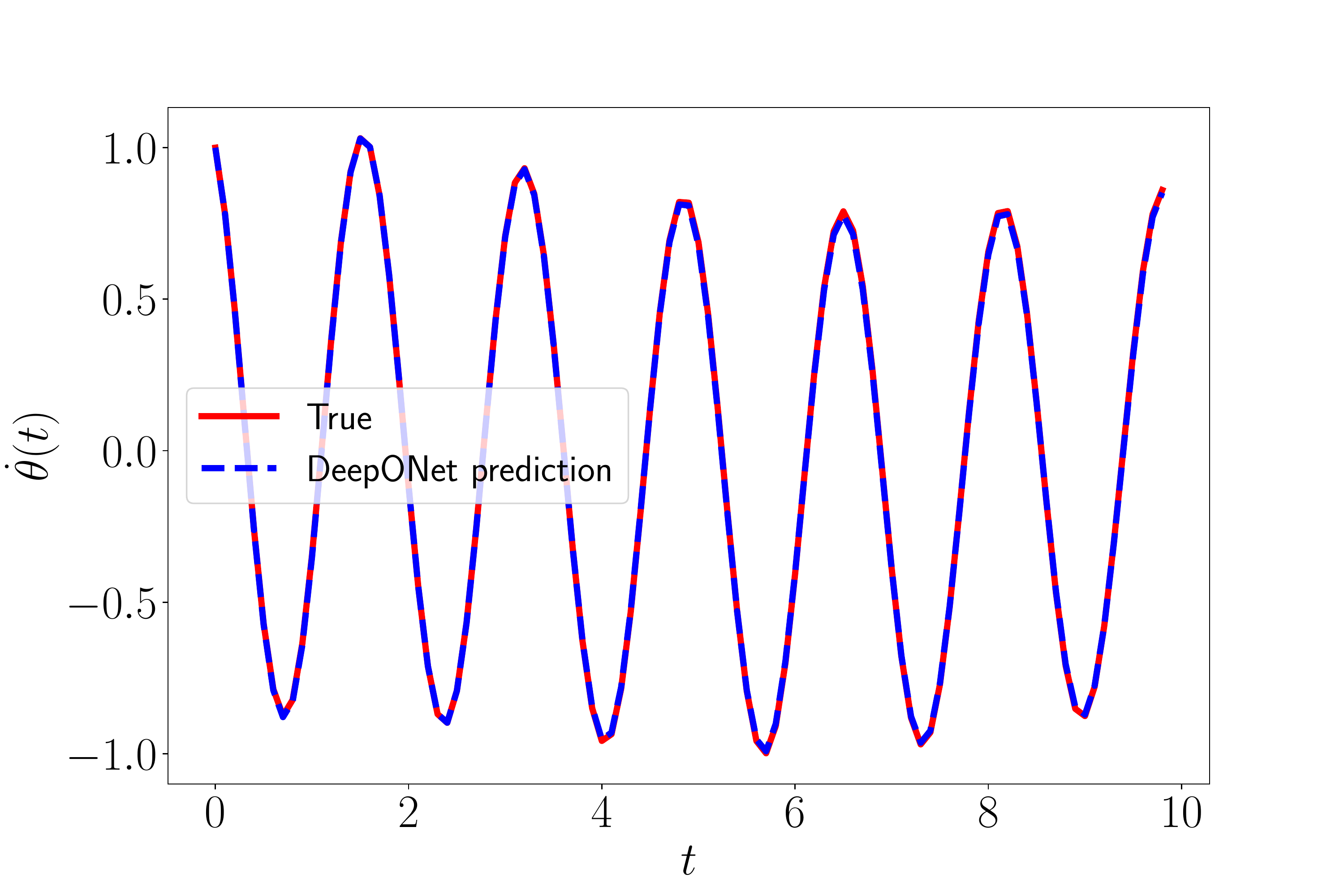}
\end{subfigure}
\caption{Comparison of DeepONet prediction with the actual trajectories of the pendulum~\eqref{eq:pendulum-swing-up} system's state~$x = (\theta(t), \dot{\theta}(t))^\top$ (\textit{left} is angle $\theta(t)$ and \textit{right} is velocity $\dot{\theta}(t)$) response to the input signal~$u(t) = \sin(t/2)$ within the partition $\mathcal{P} \subset [0,10]$~(s) of constant step size $h = 0.1$. The relative errors are presented in Table \ref{table:psu_os1_error}.}
\label{fig:pendulum-prediction-u2-h0p1}
\end{figure}

Let us now consider a different partition $\mathcal{P}'$ with a smaller step size $h=0.00025$. Figure~\ref{fig:pendulum-prediction-u2-hsmall} shows that the proposed DeepONet fails to keep up with the oscillatory response. To improve our prediction, we employ the proposed data-driven Runge-Kutta DeepONet method described in Algorithm~\ref{alg:data-driven-rk-scheme}. Figure~\ref{fig:pendulum-prediction-u2-hsmall} depicts the agreement between the actual trajectory and the trajectory predicted using the data-driven RK DeepONet method. The corresponding $L_2$-relative errors are $7.73\%$ and $11.34\%$ for the angle $\theta(t)$ and the angular velocity $\dot{\theta}(t)$, respectively.
\begin{figure}[t!]
\centering
\begin{subfigure}[b]{0.49\textwidth}
\centering
\includegraphics[width=1.1\textwidth, height=6.0cm]{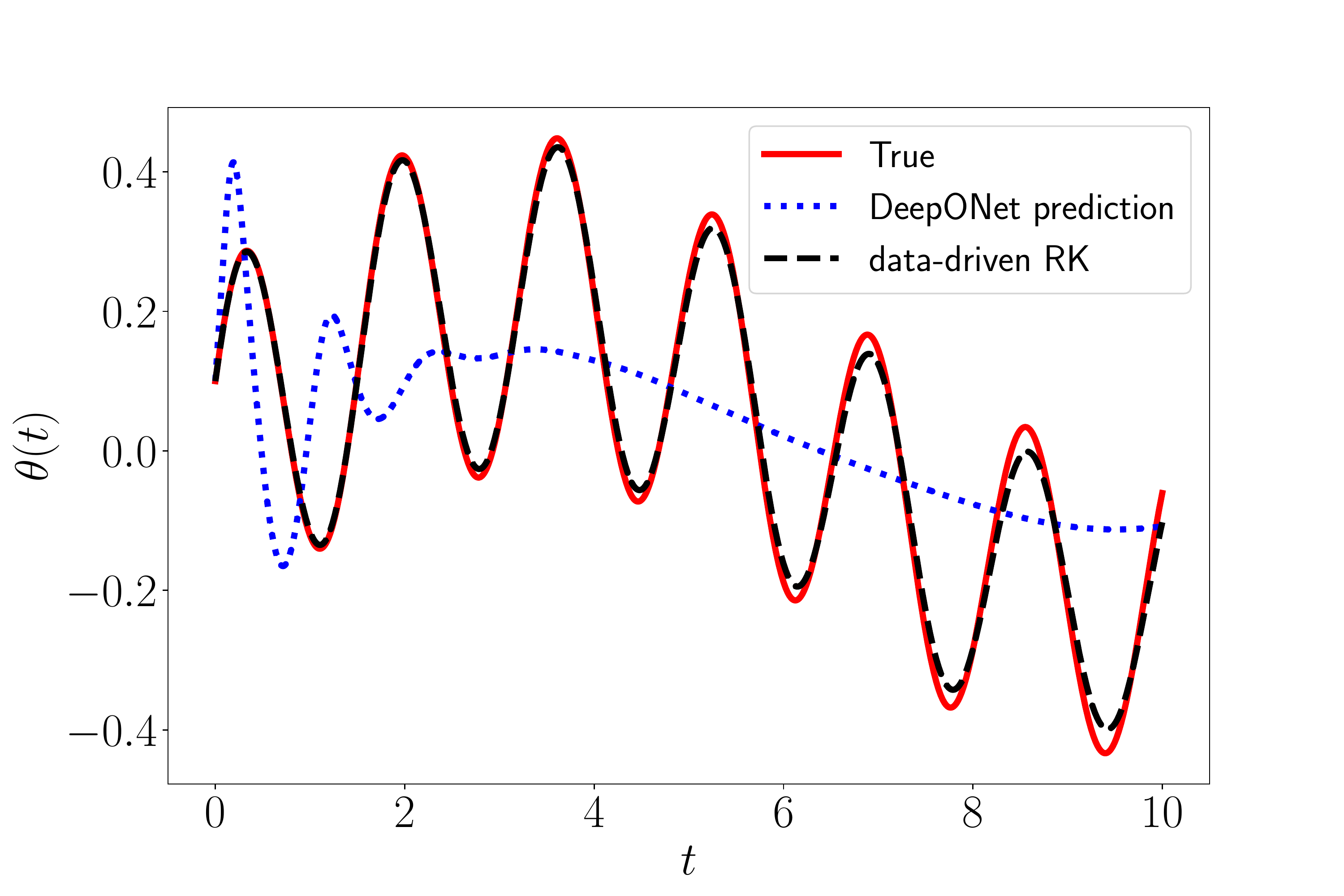}
\end{subfigure}
\begin{subfigure}[b]{0.49\textwidth}
\centering
\includegraphics[width=1.1\textwidth, height=6.0cm]{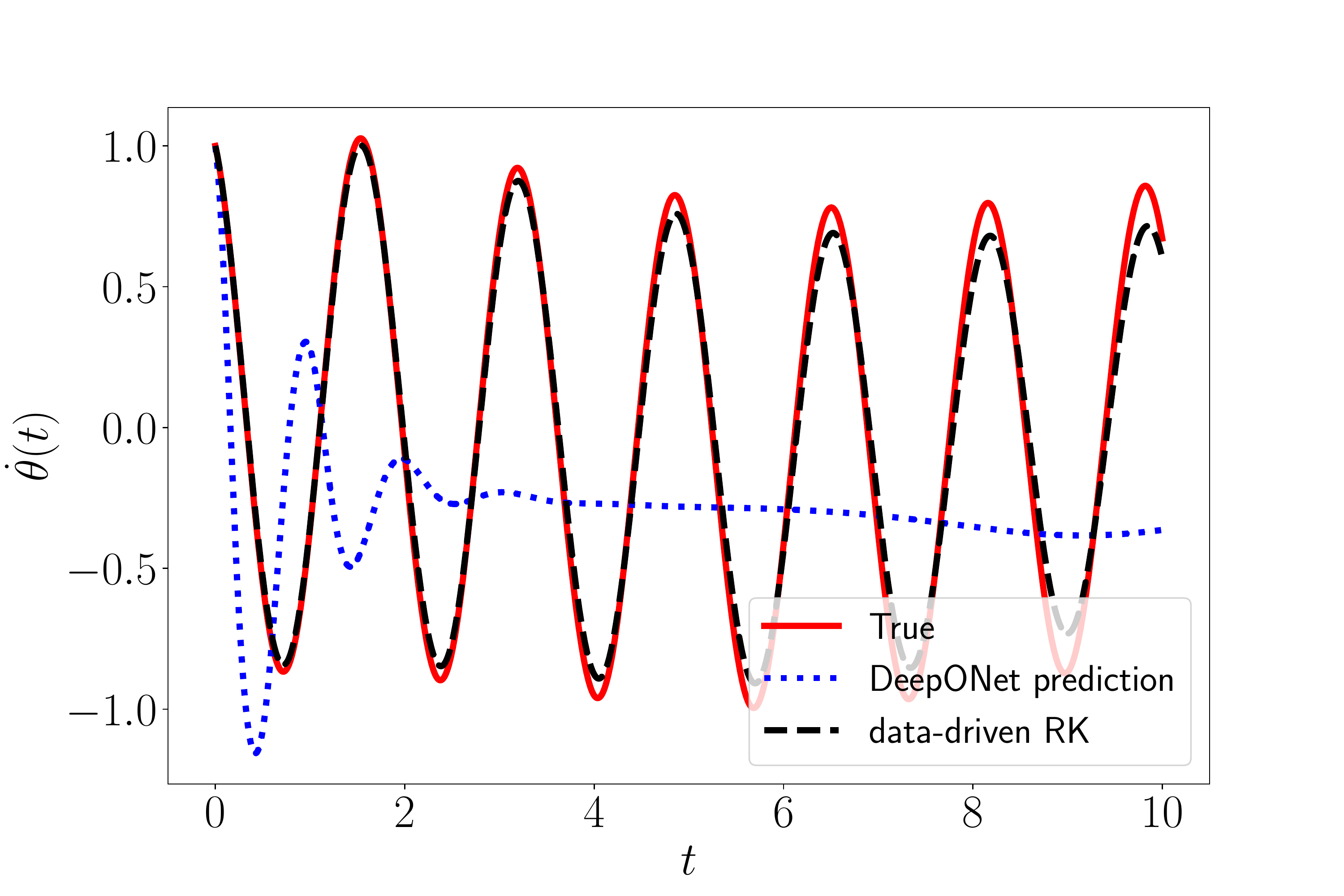}
\end{subfigure}
\caption{Comparison of DeepONet prediction with the data-driven RK DeepONet prediction, and the actual trajectories of the pendulum~\eqref{eq:pendulum-swing-up} system's state~$x = (\theta(t), \dot{\theta}(t))^\top$ (\textit{left} is angle $\theta(t)$ and \textit{right} is velocity $\dot{\theta}(t)$) response to the input signal~$u(t) = \sin(t/2)$ within the partition $\mathcal{P}' \subset [0,10]$~(s) of constant step size $h = 0.00025$.}
\label{fig:pendulum-prediction-u2-hsmall}
\end{figure}

\textit{Comparison with other benchmarks.} We compare the proposed DeepONet numerical scheme (Algorithm~\ref{alg:prediction-scheme}) with other benchmarks. The authors are only aware of the work developed in~\cite{qin2021data}, which can approximate the dynamic response of continuous nonlinear non-autonomous systems over an arbitrary time partition $\mathcal{P} \subset [0,T]$ for arbitrary input functions. Below, we will compare our method with the approach in~\cite{qin2021data}, which we refer to as FNN.

Additionally, a related field that attempts to learn an environment model of a discrete, non-linear, non-autonomous system is model-based reinforcement learning (MBRL)~\cite{wang2019benchmarking}. However, since they can only approximate the response over a uniform partition, a comparison of traditional MBRL benchmarks with the proposed work is not possible. To address this problem, we transform the ensemble method~\cite{wang2019benchmarking}, a simple but effective MBRL method to learn the environment dynamics in MBRL, into a continuous approach using ideas from~\cite{qin2021data}. Thus, we also compare the proposed method with this continuous ensemble approach.

We begin by comparing the average and standard deviation of the $L_2$ relative error between the predicted response from DeepONet, FNN, and Ensemble models. These models were all trained using $N_{\text{train}}=10000$ one-step responses, and compared with the actual response trajectories of the pendulum system~\eqref{eq:pendulum-swing-up}. The pendulum trajectories are the responses to (i) the control torque $u_1(t) = -0.8 \dot{\theta(t)}$ and (ii) $100$ initial conditions uniformly sampled from the set $\mathcal{X}_o := \{\theta, \dot{\theta}:\theta \in [-\pi/2,\pi/2], \dot{\theta} = 0\}$. Table~\ref{table:pendulum-error-comparison} shows that DeepONet greatly improves generalization when compared with FNN and the continuous ensemble.
{\begin{table}[t]
\centering
\begin{tabular}{| c | c  c|}
\hline
& angle~$\theta(t)$ & angular velocity~$\dot{\theta}(t)$\\
\hline
\hline
mean $L_2$ (DeepONet) & 0.833 \% & 1.061 \%\\
st.dev. $L_2$ (DeepONet) & 0.604 \% & 0.756\% \\
\hline
mean $L_2$ (FNN) & 10.076 \% & 13.318 \%\\
st.dev. $L_2$ (FNN) & 7.626 \% & 9.817\% \\
\hline
mean $L_2$ (Ensemble) & 5.820 \% & 7.055 \%\\
st.dev. $L_2$ (Ensemble) & 3.817  \% & 4.654\% \\
\hline
\end{tabular}
\caption{The average and standard deviation (st.dev.) of the $L_2-$ relative error between the predicted (DeepONet, FNN, and Ensemble) and actual response trajectories of the pendulum system~\eqref{eq:pendulum-swing-up} to (i) the control torque~$u_1(t) = -0.8 \dot{\theta(t)}$ and (ii) $100$ initial conditions uniformly sampled from the set $\mathcal{X}_o := \{\theta, \dot{\theta}:\theta \in [-\pi/2,\pi/2], \dot{\theta} = 0\}$.}
\label{table:pendulum-error-comparison}
\end{table}}

In addition, we compared the predictions of DeepONet, FNN, and Ensemble with the actual trajectories of the pendulum~\eqref{eq:pendulum-swing-up} system's state~$x = (\theta(t), \dot{\theta}(t))^\top$. We analyzed the response to the input signal~$u(t) = \sin(t/2)$ within the partition $\mathcal{P} \subset [0,10]$~(s) with a constant step size of $h = 0.1$. Figure~\ref{fig:pendulum-prediction-u2-h0p1-comparison}  (the \textit{left} plot of Figure~\ref{fig:pendulum-prediction-u2-h0p1-comparison} shows the angle $\theta(t)$, and the \textit{right} plot shows the velocity $\dot{\theta}(t)$) illustrates that only DeepONet can accurately predict oscillatory behavior.

\begin{figure}[t!]
\centering
\begin{subfigure}[b]{0.49\textwidth}
\centering
\includegraphics[width=1.1\textwidth, height=6.0cm]{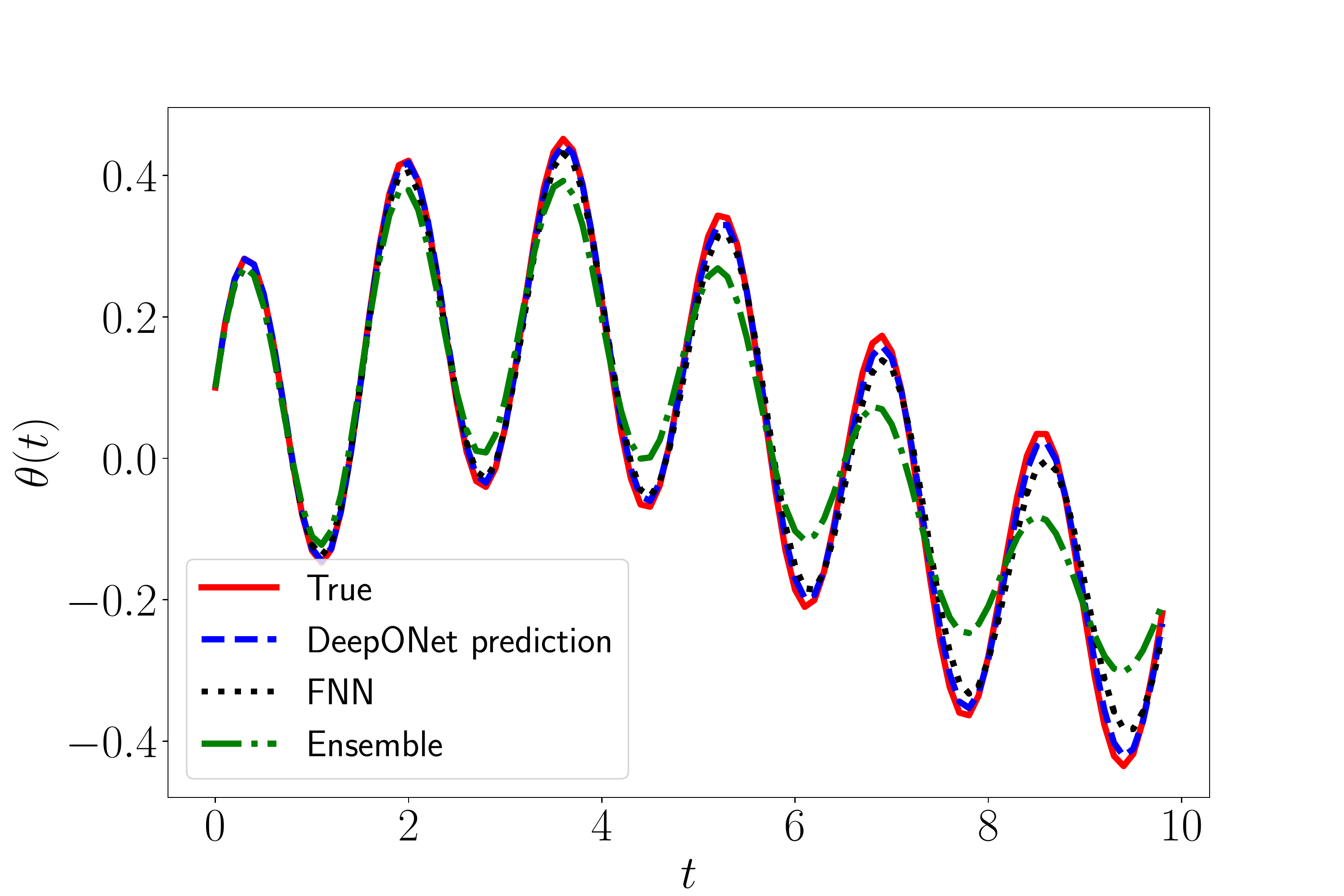}
\end{subfigure}
\begin{subfigure}[b]{0.49\textwidth}
\centering
\includegraphics[width=1.1\textwidth, height=6.0cm]{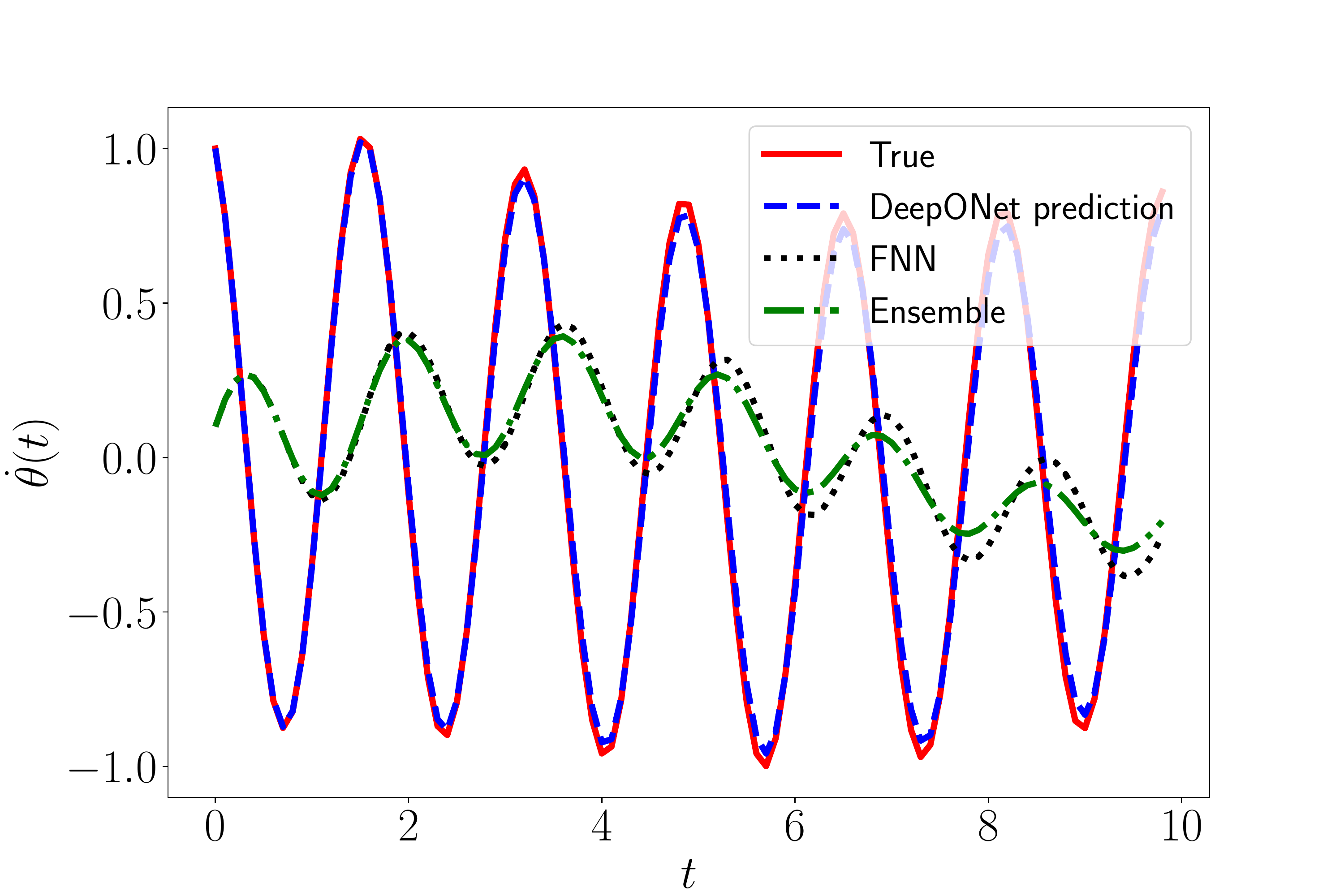}
\end{subfigure}
\caption{Comparison of DeepONet, FNN, and Ensemble prediction with the actual trajectories of the pendulum~\eqref{eq:pendulum-swing-up} system's state~$x = (\theta(t), \dot{\theta}(t))^\top$ (\textit{left} is angle $\theta(t)$ and \textit{right} is velocity $\dot{\theta}(t)$) response to the input signal~$u(t) = \sin(t/2)$ within the partition $\mathcal{P} \subset [0,10]$~(s) of constant step size $h = 0.1$.}
\label{fig:pendulum-prediction-u2-h0p1-comparison}
\end{figure}

\subsection{Cart-Pole} \label{ssec:cartpole}
In this section, we consider the following cart-pole system~\cite{florian2007correct} with control:
\begin{gather}
    \begin{aligned}
         \ddot{\theta}&= \frac{g \sin \theta + \cos \theta \left(\frac{-u(t) -m_p l \dot{\theta}^2 \sin \theta}{m_c+m_p}\right)}{l\left(\frac{4}{3} - \frac{m_p \cos^2 \theta}{m_c+m_p} \right) }\\
         \ddot{p}&= \frac{u(t) -b\dot{p} + m_pl(\dot{\theta}^2 \sin \theta - \ddot{\theta} \cos \theta)}{m_c + m_p}.
    \end{aligned}\label{eq:cartpole}
\end{gather}
In the above, the state of the cart-pole system is $x =(\theta, \dot{\theta}, p, \dot{p})^\top \in \mathcal{X}$, where $\theta$ is the angle of the pendulum, $\dot{\theta}$ the angular velocity of the pendulum, $p$ and $\dot{p}$ are, respectively, the position and the velocity of the cart. The input~$u \in \mathcal{U}$ is the horizontal force that makes the cart move to the left or the right. We selected the parameters of the cart-pole system as follows. The pendulum's length is $l = 0.5$~(m), the pendulum's mass is $m_p = 0.5$~(kg), the cart's mass is~$m_c = 0.5$~(kg), and the friction coefficient is~$b = 0.01$~(sNm/rad).

To train the DeepONet, we generated~$N_\text{train} = 20000$ trajectories with the initial condition~$x_i(t_n)$ (resp. input~$u_i(t_n)$) sampled from the state space $\mathcal{X}:=[-2\pi, 2\pi] \times [-\pi, \pi] \times [-2, 2] \times [-1,1]$ (resp. $\mathcal{U}:=[-5,5]$).

\begin{figure}[t!]
\centering
\begin{subfigure}[b]{0.49\textwidth}
\centering
\includegraphics[width=1.0\textwidth, height=5.25cm]{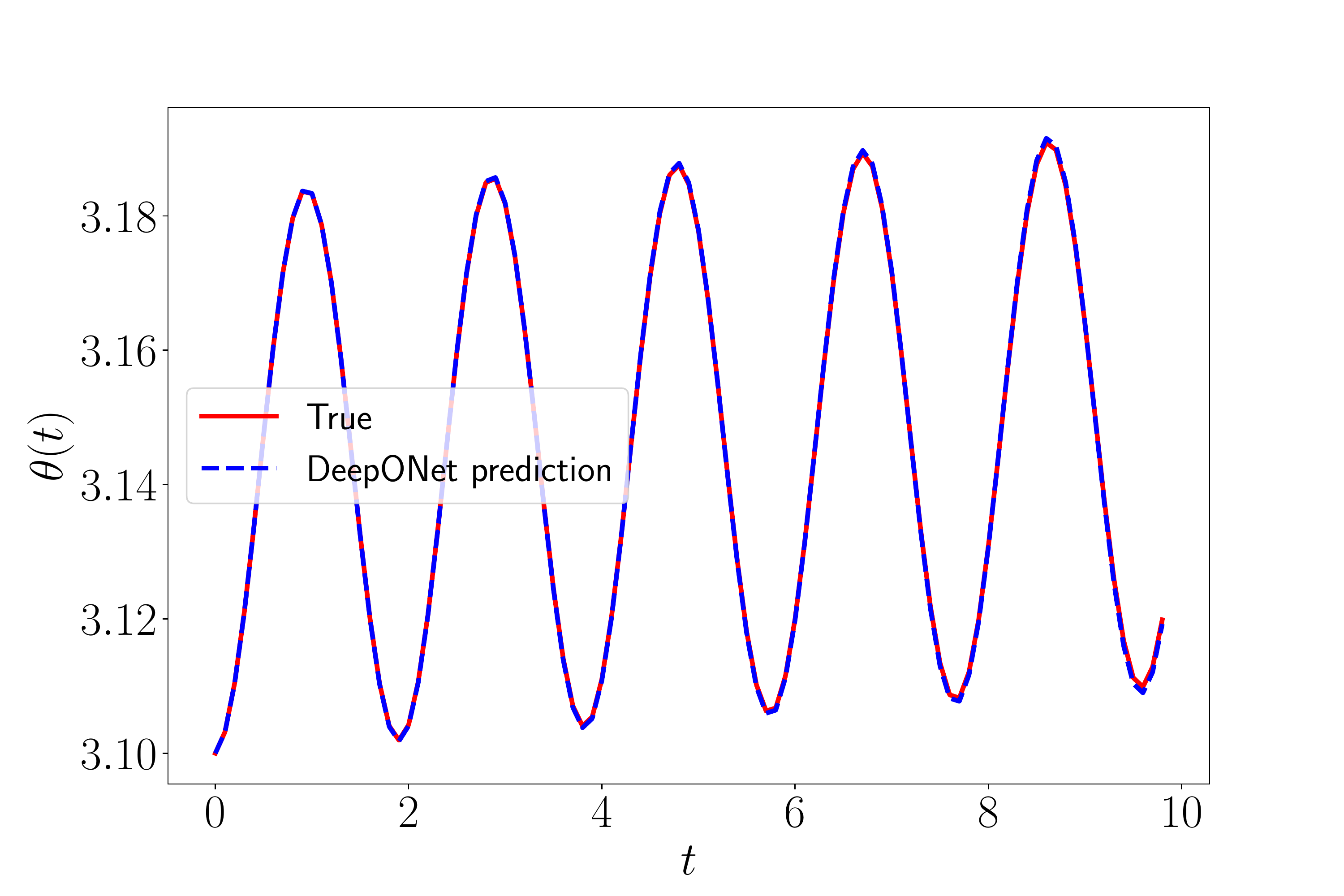}
\end{subfigure}
\begin{subfigure}[b]{0.49\textwidth}
\centering
\includegraphics[width=1.0\textwidth, height=5.25cm]{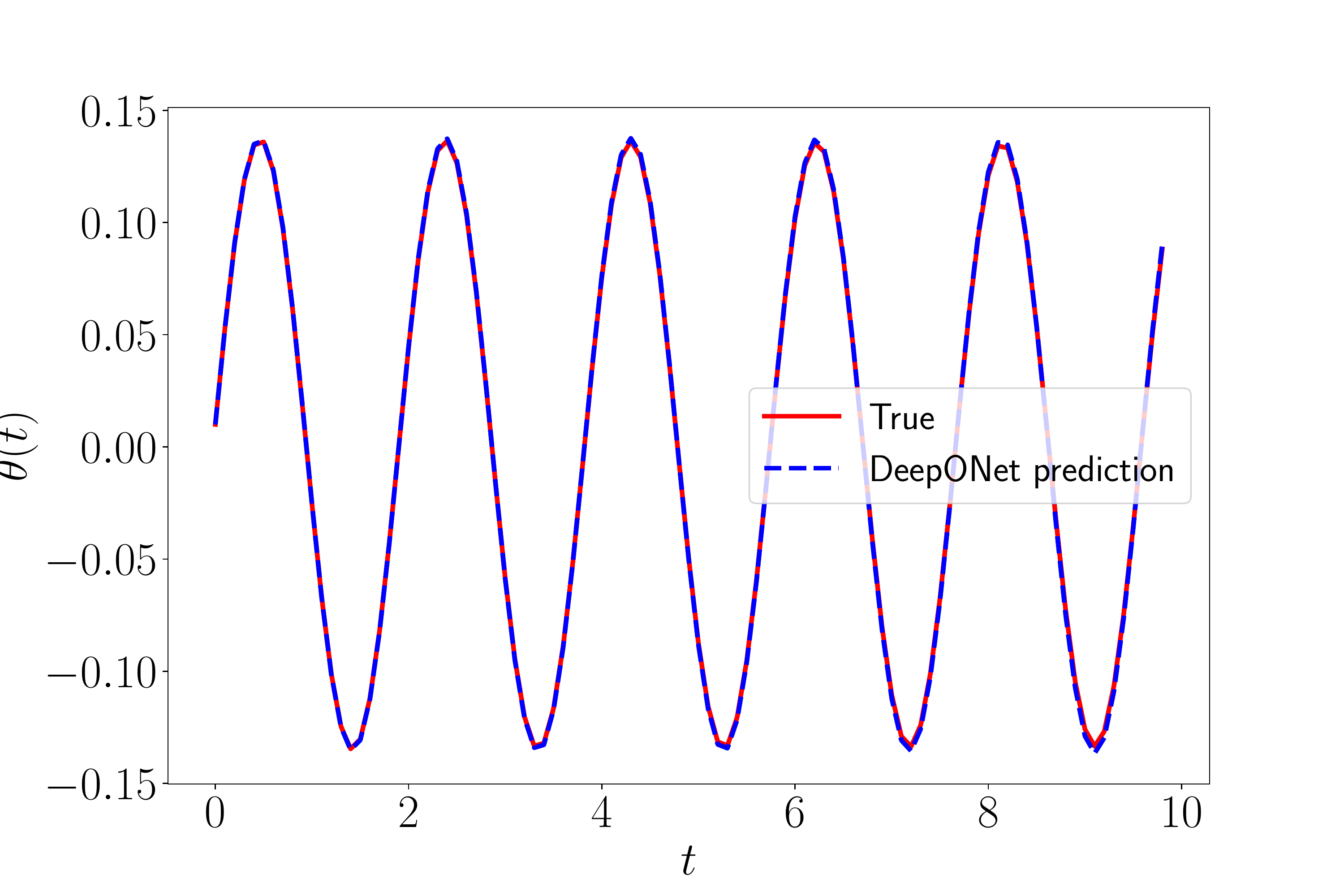}
\end{subfigure}
\\
\begin{subfigure}[b]{0.49\textwidth}
\centering
\includegraphics[width=1.0\textwidth, height=5.25cm]{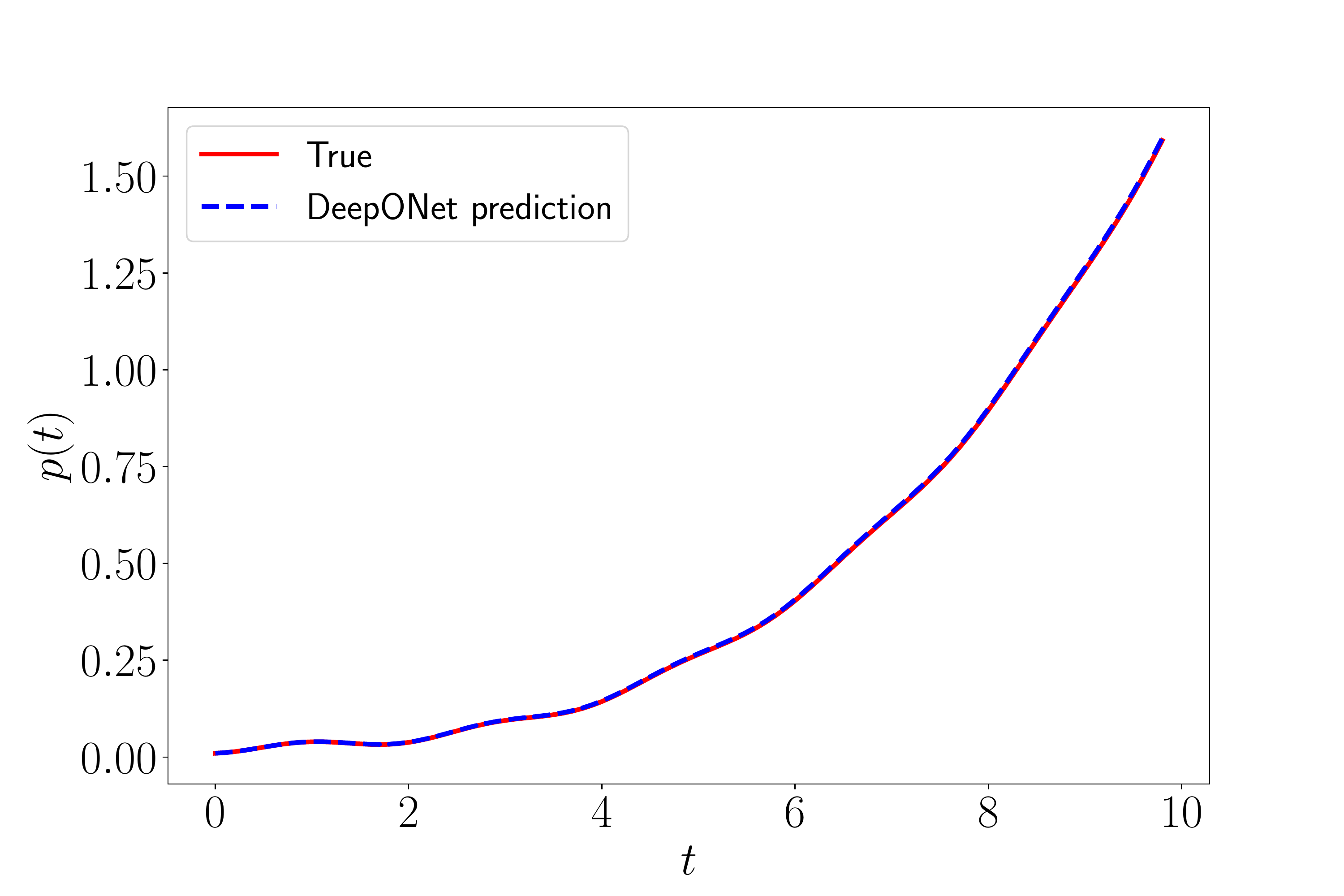}
\end{subfigure}
\begin{subfigure}[b]{0.49\textwidth}
\centering
\includegraphics[width=1.0\textwidth, height=5.25cm]{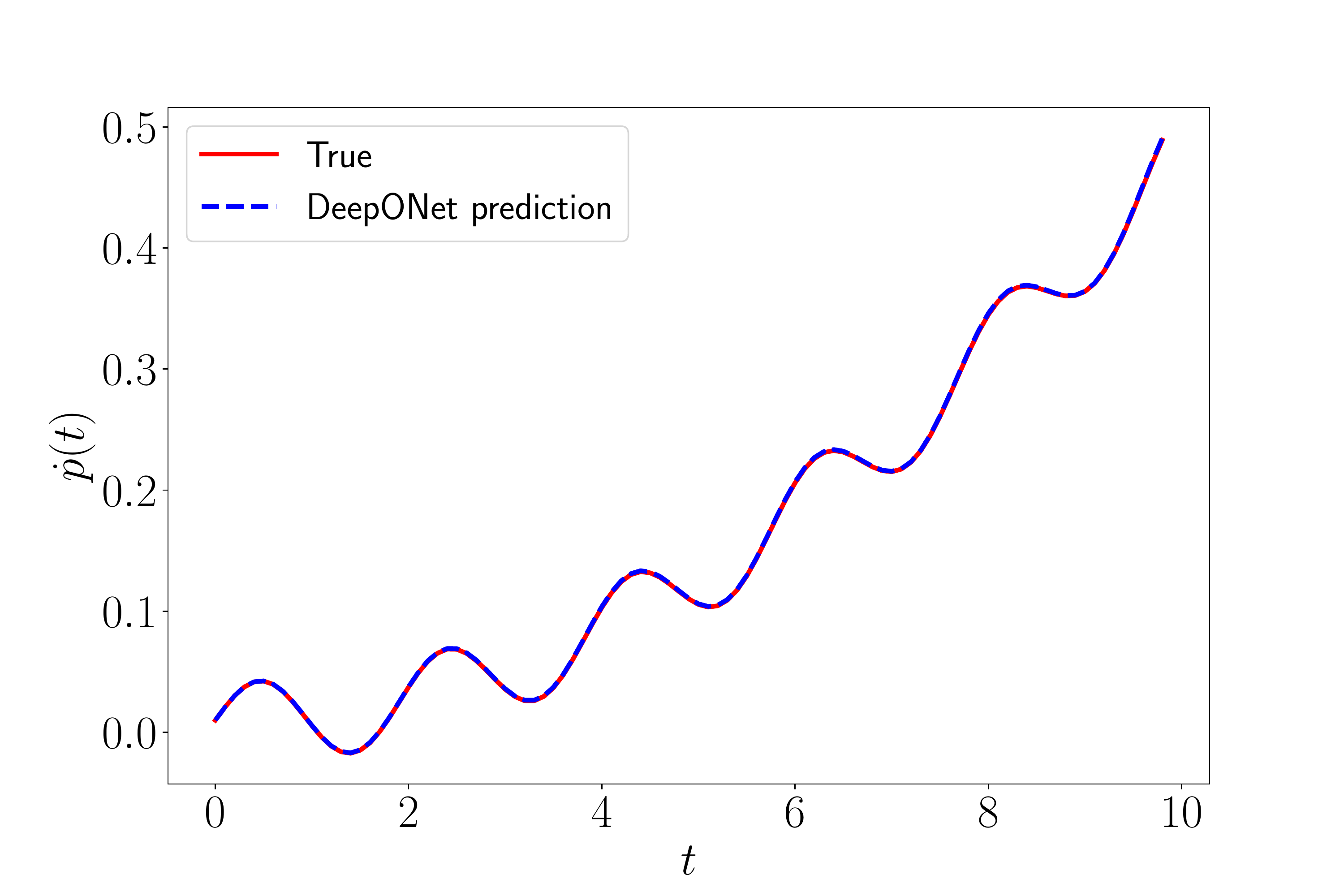}
\end{subfigure}
\caption{Comparison of DeepONet prediction with the actual trajectory of the cart-pole~\eqref{eq:cartpole} system's state~$x = (\theta(t), \dot{\theta}(t), p(t), \dot{p}(t))^\top$ (\textit{upper-left} is angle $\theta(t)$, \textit{upper-right} is angular velocity $\dot{\theta}(t)$, \textit{lower-left} is position $p(t)$, and \textit{lower-right} is velocity $\dot{p}(t)$) response to the input signal~$u(t) = t/100$ within the partition $\mathcal{P} \subset [0,10]$~(s) of constant step size $h = 0.1$.}
\label{fig:cart-pole-prediction}
\end{figure}

We used a trained DeepONet to predict the response of the cart-pole system to a time-dependent input signal $u(t) = t/100$ within the partition $\mathcal{P} \subset [0,10]$ (s) with a constant step size of $h = 0.1$ (\ie $t_{n+1} - t_n = h$) for all $t_n, t_{n+1} \in \mathcal{P}$. In Figure~\ref{fig:cart-pole-prediction}, we compare the predicted and true state trajectories. We observe a high degree of agreement between the predicted and true state trajectories, with $L_2$-relative errors of $0.008\%$, $1.028\%$, $0.478\%$, and $0.296\%$ for $\theta$, $\dot{\theta}$, $p$, and $\dot{p}$, respectively. Based on these results, we conclude that the proposed DeepONet framework is effective in predicting the response of a nonlinear control system, such as the cart-pole, to given inputs and different initial conditions.

\begin{table}[t]
\centering
\begin{tabular}{| c | c  c c c|}
\hline
& $\theta$ & $\dot{\theta}$ & $p$ & $\dot{p}$\\
\hline
\hline
$L_2$ relative error & $0.008$ \% & $1.028$\% & $0.478\%$ & $0.296\%$\\
\hline
\end{tabular}
\caption{Relative errors of the cart-pole example. The solutions are plotted in Figure \ref{fig:cart-pole-prediction}. }
\label{table:psu_os1_error}
\end{table}

\subsection{A Power Engineering Application} \label{ssec:power-eng}
The dynamics of the future power grid will fundamentally differ from today's grid due to widespread installation of energy storage systems, offshore wind power plants, and electric vehicle fast-charging sites. These devices will connect to the power grid via power electronic devices. This scenario differs greatly from the current paradigm where robust models of traditional power systems exist. Thus, simulating future power grids for planning, optimization, and control will require the interplay of legacy simulators and models of power electronics-based devices (\eg wind turbines) that can be learned from data.

In this final experiment, we aim to demonstrate that the proposed method can potentially interact with a traditional power engineering simulator, such as the Power System Toolbox. To achieve this, we begin with a simple application of recursive DeepONet that uses Algorithm~\ref{alg:prediction-scheme} to approximate the dynamic response of a simple second order model of a generator. We use data collected with the Power System Toolbox (PST)~\cite{chow2000power} of generator 1 within the classic two area system. Note that the response of the generator can be modeled as a non-autonomous system~\eqref{eq:control-system} by assuming the input function corresponds to the interface variables (stator currents) of the bus where the target generator connects.

\textit{Training data.} We generated training data $\mathcal{D}_{\text{PST}}$ by simulating $N_\text{exp}=600$ disturbance experiments on PST. Each experiment consisted of simulating the two area system system using a uniform partition $\mathcal{P} \subset [0.0,5.0]$ (s), with a constant step size of $h = 0.005$. The disturbances occur at $t_f=0.01$ (s) with a duration sampled from the interval $[0.005, 0.02]$ (s). Trajectory data was collected after each experiment, including the interacting input trajectories $\{u(t_n):t_n \in \mathcal{P}\}$ and state trajectory data $\{x(t_n):t_n \in \mathcal{P}\}$.

We constructed the training dataset by interpolating and sampling this trajectory data (\ie a time-series). In particular, we discretized the inputs using interpolation and $m=2$ sensors, \ie $\tilde{u}^n_m :=\{u(t_n+d_0), u(t_n + d_1))\}$, where $d_0=0.0$ and $d_1$ was uniformly sampled from the open interval $(0,h)$. The final training dataset of size $N_\text{train}=30000$ is:
$\mathcal{D}_\text{PST} = \{(x_i(t_n),\tilde{u}^n_{m,i}), \{0,d_{m,i}\}, h_i, x_i(t_n+h_i)\}_{i=1}^{N_\text{train}}.$

We tested the proposed DeepONet using a PST test trajectory that was not included in the training dataset and which experienced a disturbance of duration 0.01 (s). For these test trajectory, we used an uniform partition $\mathcal{P}$ of size $100$. As shown in Figure~\ref{fig:pst-response}, DeepONet-based Algorithm~\ref{alg:prediction-scheme} accurately predicted the dynamic response of generator 1 in the classic two area system.

\begin{figure}[t!]
\centering
\begin{subfigure}[b]{0.49\textwidth}
\centering
\includegraphics[width=1.0\textwidth, height=6.0cm]{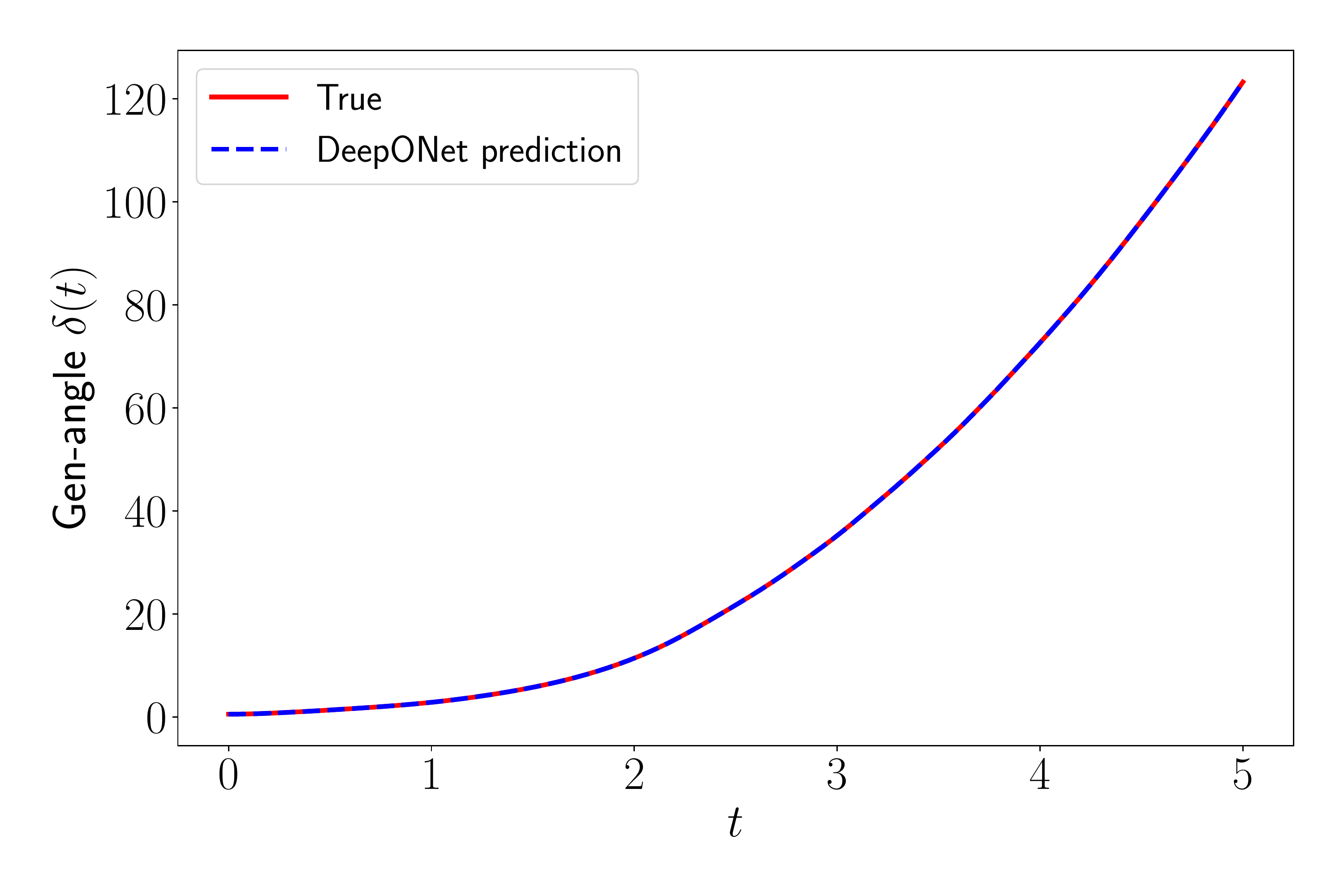}
\end{subfigure}
\begin{subfigure}[b]{0.49\textwidth}
\centering
\includegraphics[width=1.0\textwidth, height=6.0cm]{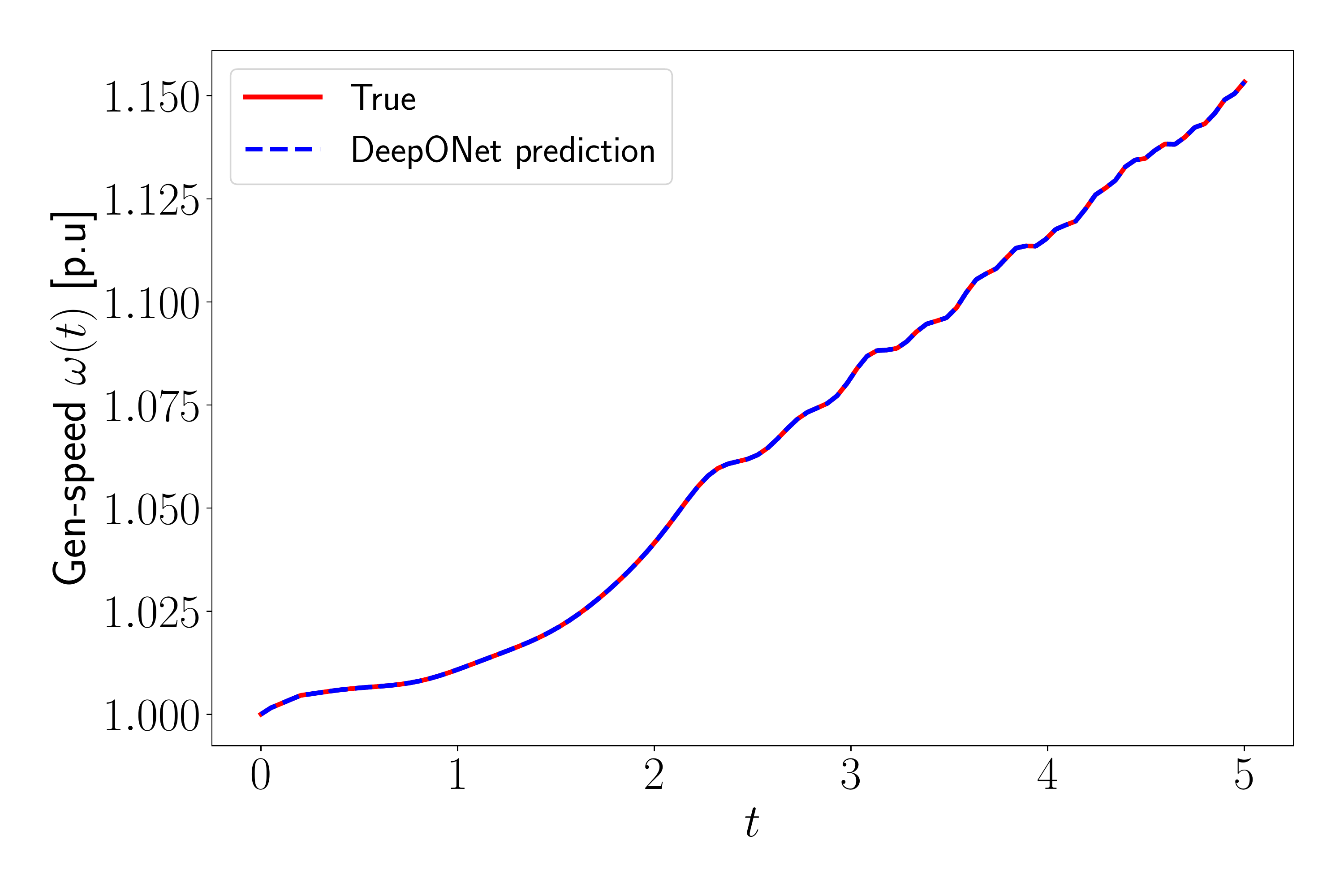}
\end{subfigure}
\caption{Comparison of DeepONet prediction with the actual trajectory of a second-order electromagnetic generator model state~$x = (\delta(t), \omega(t))^\top$ (\textit{left} is $\delta(t)$ angle and \textit{right} is $\omega(t)$ speed) for a disturbance of duration $0.01$ (s) within a uniform partition $\mathcal{P} \subset [0,5]$~(s) size $100$.}
\label{fig:pst-response}
\end{figure}

In our future work, we plan to (i) extend this approach to allow for the full interaction of DeepONets and traditional numerical schemes, (ii) enable transfer learning for approximating more than one generator, and (iii) design learning protocols that can withstand error accumulation when predicting more complex and severe discontinuous disturbances.

\section{Discussion} \label{sec:discussion}
{This section discusses our main results and outlines our plans for future work.}

{We have demonstrated through all five tasks outlined in Section~\ref{sec:numerical-experiments} that the proposed recursive and data-driven RK DeepONets can effectively approximate the dynamic response of nonlinear, non-autonomous systems for varying inputs and initial conditions. Furthermore, in Section~\ref{ssec:pendulum}, we compared the proposed recursive method with two other benchmarks. The first is a neural network approach for approximating non-autonomous systems developed in~\cite{qin2021data}. The second is an ensemble method~\cite{wang2019benchmarking} that we extended to the continuous scenario for comparison purposes. Table~\ref{table:pendulum-error-comparison} shows that the proposed recursive DeepONet outperforms the other two benchmarks when the training dataset size is small. In particular, the recursive method keeps the mean $L_2$-relative error for 100 test trajectories below $1.1$~\%. Additionally, Figure~\ref{fig:pendulum-prediction-u2-h0p1-comparison} demonstrates that among all benchmarks, the proposed recursive DeepONet is the only one that can effectively approximate challenging oscillatory solution trajectories. In conclusion, when the dataset is small, which often occurs in engineering systems, the proposed recursive DeepONet provides the best method for approximating the dynamic response of non-autonomous systems.}

{In \textit{our future work}, we aim to extend the proposed framework to include (i) reduced-order, (ii) stochastic, and (iii) networked non-autonomous and control systems. We also plan to apply the DeepONet framework to model-based reinforcement learning and control design. Specifically, we aim to use it for learning semi-Markov decision processes, which can be applied to learn suboptimal offline policies. Moreover, we aim to apply the recursive and data-driven RK DeepONets to complex engineering applications such as fluid dynamics, materials engineering, and future power systems.}
\section{Conclusion} \label{sec:conclusion}
We introduced a Deep Operator Network (DeepONet) framework that can learn (from data) the dynamic response of nonlinear, non-autonomous systems with time-dependent inputs for long-term horizons. The proposed framework approximates the system's solution operator locally using the DeepONet and then recursively predicts the system's response for long/medium-term horizons using the trained network. We also estimated the error bound for this DeepONet-based numerical scheme. To improve the predictive accuracy of the scheme when the step size is small, we designed and theoretically validated a data-driven Runge-Kutta DeepONet scheme. This scheme uses estimates of the vector field computed with the DeepONet forward pass and automatic differentiation. Finally, we validated the proposed framework using an autonomous and chaotic system, three continuous control tasks, and a power engineering application.

\section*{Acknowledgments}
We gratefully acknowledge the support of the National Science Foundation (DMS-1555072, DMS-2053746, and DMS-2134209), Brookhaven National Laboratory Subcontract 382247, and U.S. Department of Energy (DOE) Office of Science Advanced Scientific Computing Research program DE-SC0021142 and DE-SC0023161.

\section*{Declaration of Competing Interest}
The authors declare that they have no known competing financial interests or personal relationships that could have appeared to influence the work reported in this paper.

\bibliographystyle{elsarticle-harv} 
\bibliography{references}

\begin{thebibliography}{43}
\expandafter\ifx\csname natexlab\endcsname\relax\def\natexlab#1{#1}\fi
\providecommand{\url}[1]{\texttt{#1}}
\providecommand{\href}[2]{#2}
\providecommand{\path}[1]{#1}
\providecommand{\DOIprefix}{doi:}
\providecommand{\ArXivprefix}{arXiv:}
\providecommand{\URLprefix}{URL: }
\providecommand{\Pubmedprefix}{pmid:}
\providecommand{\doi}[1]{\href{http://dx.doi.org/#1}{\path{#1}}}
\providecommand{\Pubmed}[1]{\href{pmid:#1}{\path{#1}}}
\providecommand{\bibinfo}[2]{#2}
\ifx\xfnm\relax \def\xfnm[#1]{\unskip,\space#1}\fi
\bibitem[{Bradbury et~al.(2018)Bradbury, Frostig, Hawkins, Johnson, Leary,
  Maclaurin, Necula, Paszke, Vander{P}las, Wanderman-{M}ilne and
  Zhang}]{jax2018github}
\bibinfo{author}{Bradbury, J.}, \bibinfo{author}{Frostig, R.},
  \bibinfo{author}{Hawkins, P.}, \bibinfo{author}{Johnson, M.J.},
  \bibinfo{author}{Leary, C.}, \bibinfo{author}{Maclaurin, D.},
  \bibinfo{author}{Necula, G.}, \bibinfo{author}{Paszke, A.},
  \bibinfo{author}{Vander{P}las, J.}, \bibinfo{author}{Wanderman-{M}ilne, S.},
  \bibinfo{author}{Zhang, Q.}, \bibinfo{year}{2018}.
\newblock \bibinfo{title}{{JAX}: composable transformations of
  {P}ython+{N}um{P}y programs}.
\newblock \URLprefix \url{http://github.com/google/jax}.
\bibitem[{Brockman et~al.(2016)Brockman, Cheung, Pettersson, Schneider,
  Schulman, Tang and Zaremba}]{brockman2016openai}
\bibinfo{author}{Brockman, G.}, \bibinfo{author}{Cheung, V.},
  \bibinfo{author}{Pettersson, L.}, \bibinfo{author}{Schneider, J.},
  \bibinfo{author}{Schulman, J.}, \bibinfo{author}{Tang, J.},
  \bibinfo{author}{Zaremba, W.}, \bibinfo{year}{2016}.
\newblock \bibinfo{title}{Openai gym}.
\newblock \bibinfo{journal}{arXiv preprint arXiv:1606.01540} .
\bibitem[{Brunton et~al.(2016a)Brunton, Proctor and
  Kutz}]{brunton2016discovering}
\bibinfo{author}{Brunton, S.L.}, \bibinfo{author}{Proctor, J.L.},
  \bibinfo{author}{Kutz, J.N.}, \bibinfo{year}{2016}a.
\newblock \bibinfo{title}{Discovering governing equations from data by sparse
  identification of nonlinear dynamical systems}.
\newblock \bibinfo{journal}{Proceedings of the national academy of sciences}
  \bibinfo{volume}{113}, \bibinfo{pages}{3932--3937}.
\bibitem[{Brunton et~al.(2016b)Brunton, Proctor and Kutz}]{brunton2016sparse}
\bibinfo{author}{Brunton, S.L.}, \bibinfo{author}{Proctor, J.L.},
  \bibinfo{author}{Kutz, J.N.}, \bibinfo{year}{2016}b.
\newblock \bibinfo{title}{Sparse identification of nonlinear dynamics with
  control (sindyc)}.
\newblock \bibinfo{journal}{IFAC-PapersOnLine} \bibinfo{volume}{49},
  \bibinfo{pages}{710--715}.
\bibitem[{Cai et~al.(2021)Cai, Wang, Lu, Zaki and Karniadakis}]{cai2021deepm}
\bibinfo{author}{Cai, S.}, \bibinfo{author}{Wang, Z.}, \bibinfo{author}{Lu,
  L.}, \bibinfo{author}{Zaki, T.A.}, \bibinfo{author}{Karniadakis, G.E.},
  \bibinfo{year}{2021}.
\newblock \bibinfo{title}{Deepm\&mnet: Inferring the electroconvection
  multiphysics fields based on operator approximation by neural networks}.
\newblock \bibinfo{journal}{Journal of Computational Physics}
  \bibinfo{volume}{436}, \bibinfo{pages}{110296}.
\bibitem[{Chen and Chen(1995)}]{chen1995universal}
\bibinfo{author}{Chen, T.}, \bibinfo{author}{Chen, H.}, \bibinfo{year}{1995}.
\newblock \bibinfo{title}{Universal approximation to nonlinear operators by
  neural networks with arbitrary activation functions and its application to
  dynamical systems}.
\newblock \bibinfo{journal}{IEEE Transactions on Neural Networks}
  \bibinfo{volume}{6}, \bibinfo{pages}{911--917}.
\bibitem[{Choudhary et~al.(2023)Choudhary, Mishra, Fatima and
  Panigrahi}]{choudhary2023multi}
\bibinfo{author}{Choudhary, A.}, \bibinfo{author}{Mishra, R.K.},
  \bibinfo{author}{Fatima, S.}, \bibinfo{author}{Panigrahi, B.},
  \bibinfo{year}{2023}.
\newblock \bibinfo{title}{Multi-input cnn based vibro-acoustic fusion for
  accurate fault diagnosis of induction motor}.
\newblock \bibinfo{journal}{Engineering Applications of Artificial
  Intelligence} \bibinfo{volume}{120}, \bibinfo{pages}{105872}.
\bibitem[{Chow et~al.(2000)Chow, Rogers and Cheung}]{chow2000power}
\bibinfo{author}{Chow, J.}, \bibinfo{author}{Rogers, G.},
  \bibinfo{author}{Cheung, K.}, \bibinfo{year}{2000}.
\newblock \bibinfo{title}{Power system toolbox}.
\newblock \bibinfo{journal}{Cherry Tree Scientific Software}
  \bibinfo{volume}{48}, \bibinfo{pages}{53}.
\bibitem[{Chung et~al.(2021)Chung, Leung, Pun and Zhang}]{chung2021multi}
\bibinfo{author}{Chung, E.}, \bibinfo{author}{Leung, W.T.},
  \bibinfo{author}{Pun, S.M.}, \bibinfo{author}{Zhang, Z.},
  \bibinfo{year}{2021}.
\newblock \bibinfo{title}{A multi-stage deep learning based algorithm for
  multiscale model reduction}.
\newblock \bibinfo{journal}{Journal of Computational and Applied Mathematics}
  \bibinfo{volume}{394}, \bibinfo{pages}{113506}.
\bibitem[{Cybenko(1989)}]{cybenko1989approximation}
\bibinfo{author}{Cybenko, G.}, \bibinfo{year}{1989}.
\newblock \bibinfo{title}{Approximation by superpositions of a sigmoidal
  function}.
\newblock \bibinfo{journal}{Mathematics of control, signals and systems}
  \bibinfo{volume}{2}, \bibinfo{pages}{303--314}.
\bibitem[{Deisenroth et~al.(2013)Deisenroth, Fox and
  Rasmussen}]{deisenroth2013gaussian}
\bibinfo{author}{Deisenroth, M.P.}, \bibinfo{author}{Fox, D.},
  \bibinfo{author}{Rasmussen, C.E.}, \bibinfo{year}{2013}.
\newblock \bibinfo{title}{Gaussian processes for data-efficient learning in
  robotics and control}.
\newblock \bibinfo{journal}{IEEE transactions on pattern analysis and machine
  intelligence} \bibinfo{volume}{37}, \bibinfo{pages}{408--423}.
\bibitem[{Du et~al.(2020)Du, Futoma and Doshi-Velez}]{du2020model}
\bibinfo{author}{Du, J.}, \bibinfo{author}{Futoma, J.},
  \bibinfo{author}{Doshi-Velez, F.}, \bibinfo{year}{2020}.
\newblock \bibinfo{title}{Model-based reinforcement learning for semi-markov
  decision processes with neural odes}.
\newblock \bibinfo{journal}{Advances in Neural Information Processing Systems}
  \bibinfo{volume}{33}, \bibinfo{pages}{19805--19816}.
\bibitem[{Efendiev et~al.(2021)Efendiev, Leung, Lin and
  Zhang}]{efendiev2021hei}
\bibinfo{author}{Efendiev, Y.}, \bibinfo{author}{Leung, W.T.},
  \bibinfo{author}{Lin, G.}, \bibinfo{author}{Zhang, Z.}, \bibinfo{year}{2021}.
\newblock \bibinfo{title}{Hei: hybrid explicit-implicit learning for multiscale
  problems}.
\newblock \bibinfo{journal}{arXiv preprint arXiv:2109.02147} .
\bibitem[{Fan et~al.(2020)Fan, Yang, Wang, Triantafyllou and
  Karniadakis}]{fan2020reinforcement}
\bibinfo{author}{Fan, D.}, \bibinfo{author}{Yang, L.}, \bibinfo{author}{Wang,
  Z.}, \bibinfo{author}{Triantafyllou, M.S.}, \bibinfo{author}{Karniadakis,
  G.E.}, \bibinfo{year}{2020}.
\newblock \bibinfo{title}{Reinforcement learning for bluff body active flow
  control in experiments and simulations}.
\newblock \bibinfo{journal}{Proceedings of the National Academy of Sciences}
  \bibinfo{volume}{117}, \bibinfo{pages}{26091--26098}.
\bibitem[{Florian(2007)}]{florian2007correct}
\bibinfo{author}{Florian, R.V.}, \bibinfo{year}{2007}.
\newblock \bibinfo{title}{Correct equations for the dynamics of the cart-pole
  system}.
\newblock \bibinfo{journal}{Center for Cognitive and Neural Studies (Coneural),
  Romania} .
\bibitem[{Hafner et~al.(2019)Hafner, Lillicrap, Ba and
  Norouzi}]{hafner2019dream}
\bibinfo{author}{Hafner, D.}, \bibinfo{author}{Lillicrap, T.},
  \bibinfo{author}{Ba, J.}, \bibinfo{author}{Norouzi, M.},
  \bibinfo{year}{2019}.
\newblock \bibinfo{title}{Dream to control: Learning behaviors by latent
  imagination}.
\newblock \bibinfo{journal}{arXiv preprint arXiv:1912.01603} .
\bibitem[{Iserles(2009)}]{iserles2009first}
\bibinfo{author}{Iserles, A.}, \bibinfo{year}{2009}.
\newblock \bibinfo{title}{A first course in the numerical analysis of
  differential equations}.
\newblock \bibinfo{number}{44}, \bibinfo{publisher}{Cambridge university
  press}.
\bibitem[{Kaiser et~al.(2019)Kaiser, Babaeizadeh, Milos, Osinski, Campbell,
  Czechowski, Erhan, Finn, Kozakowski, Levine et~al.}]{kaiser2019model}
\bibinfo{author}{Kaiser, L.}, \bibinfo{author}{Babaeizadeh, M.},
  \bibinfo{author}{Milos, P.}, \bibinfo{author}{Osinski, B.},
  \bibinfo{author}{Campbell, R.H.}, \bibinfo{author}{Czechowski, K.},
  \bibinfo{author}{Erhan, D.}, \bibinfo{author}{Finn, C.},
  \bibinfo{author}{Kozakowski, P.}, \bibinfo{author}{Levine, S.}, et~al.,
  \bibinfo{year}{2019}.
\newblock \bibinfo{title}{Model-based reinforcement learning for atari}.
\newblock \bibinfo{journal}{arXiv preprint arXiv:1903.00374} .
\bibitem[{Kingma and Ba(2014)}]{kingma2014adam}
\bibinfo{author}{Kingma, D.P.}, \bibinfo{author}{Ba, J.}, \bibinfo{year}{2014}.
\newblock \bibinfo{title}{Adam: A method for stochastic optimization}.
\newblock \bibinfo{journal}{arXiv preprint arXiv:1412.6980} .
\bibitem[{Li et~al.(2020)Li, Kovachki, Azizzadenesheli, Liu, Bhattacharya,
  Stuart and Anandkumar}]{li2020fourier}
\bibinfo{author}{Li, Z.}, \bibinfo{author}{Kovachki, N.},
  \bibinfo{author}{Azizzadenesheli, K.}, \bibinfo{author}{Liu, B.},
  \bibinfo{author}{Bhattacharya, K.}, \bibinfo{author}{Stuart, A.},
  \bibinfo{author}{Anandkumar, A.}, \bibinfo{year}{2020}.
\newblock \bibinfo{title}{Fourier neural operator for parametric partial
  differential equations}.
\newblock \bibinfo{journal}{arXiv preprint arXiv:2010.08895} .
\bibitem[{Lin et~al.(2023)Lin, Moya and Zhang}]{lin2023b}
\bibinfo{author}{Lin, G.}, \bibinfo{author}{Moya, C.}, \bibinfo{author}{Zhang,
  Z.}, \bibinfo{year}{2023}.
\newblock \bibinfo{title}{B-deeponet: An enhanced bayesian deeponet for solving
  noisy parametric pdes using accelerated replica exchange sgld}.
\newblock \bibinfo{journal}{Journal of Computational Physics}
  \bibinfo{volume}{473}, \bibinfo{pages}{111713}.
\bibitem[{Lin et~al.(2021)Lin, Wang and Zhang}]{lin2021multi}
\bibinfo{author}{Lin, G.}, \bibinfo{author}{Wang, Y.}, \bibinfo{author}{Zhang,
  Z.}, \bibinfo{year}{2021}.
\newblock \bibinfo{title}{Multi-variance replica exchange stochastic gradient
  mcmc for inverse and forward bayesian physics-informed neural network}.
\newblock \bibinfo{journal}{arXiv preprint arXiv:2107.06330} .
\bibitem[{Lu et~al.(2019)Lu, Jin and Karniadakis}]{lu2019deeponet}
\bibinfo{author}{Lu, L.}, \bibinfo{author}{Jin, P.},
  \bibinfo{author}{Karniadakis, G.E.}, \bibinfo{year}{2019}.
\newblock \bibinfo{title}{Deeponet: Learning nonlinear operators for
  identifying differential equations based on the universal approximation
  theorem of operators}.
\newblock \bibinfo{journal}{arXiv preprint arXiv:1910.03193} .
\bibitem[{Lu et~al.(2021)Lu, Jin, Pang, Zhang and Karniadakis}]{lu2021learning}
\bibinfo{author}{Lu, L.}, \bibinfo{author}{Jin, P.}, \bibinfo{author}{Pang,
  G.}, \bibinfo{author}{Zhang, Z.}, \bibinfo{author}{Karniadakis, G.E.},
  \bibinfo{year}{2021}.
\newblock \bibinfo{title}{Learning nonlinear operators via deeponet based on
  the universal approximation theorem of operators}.
\newblock \bibinfo{journal}{Nature Machine Intelligence} \bibinfo{volume}{3},
  \bibinfo{pages}{218--229}.
\bibitem[{Mishra et~al.(2022a)Mishra, Choudhary, Fatima, Mohanty and
  Panigrahi}]{mishra2022fault}
\bibinfo{author}{Mishra, R.}, \bibinfo{author}{Choudhary, A.},
  \bibinfo{author}{Fatima, S.}, \bibinfo{author}{Mohanty, A.},
  \bibinfo{author}{Panigrahi, B.}, \bibinfo{year}{2022}a.
\newblock \bibinfo{title}{A fault diagnosis approach based on 2d-vibration
  imaging for bearing faults}.
\newblock \bibinfo{journal}{Journal of Vibration Engineering \& Technologies} ,
  \bibinfo{pages}{1--14}.
\bibitem[{Mishra et~al.(2022b)Mishra, Choudhary, Fatima, Mohanty and
  Panigrahi}]{mishra2022self}
\bibinfo{author}{Mishra, R.K.}, \bibinfo{author}{Choudhary, A.},
  \bibinfo{author}{Fatima, S.}, \bibinfo{author}{Mohanty, A.R.},
  \bibinfo{author}{Panigrahi, B.K.}, \bibinfo{year}{2022}b.
\newblock \bibinfo{title}{A self-adaptive multiple-fault diagnosis system for
  rolling element bearings}.
\newblock \bibinfo{journal}{Measurement Science and Technology}
  \bibinfo{volume}{33}, \bibinfo{pages}{125018}.
\bibitem[{Moya and Lin(2023)}]{moya2023dae}
\bibinfo{author}{Moya, C.}, \bibinfo{author}{Lin, G.}, \bibinfo{year}{2023}.
\newblock \bibinfo{title}{Dae-pinn: a physics-informed neural network model for
  simulating differential algebraic equations with application to power
  networks}.
\newblock \bibinfo{journal}{Neural Computing and Applications}
  \bibinfo{volume}{35}, \bibinfo{pages}{3789--3804}.
\bibitem[{Moya et~al.(2023)Moya, Zhang, Lin and Yue}]{moya2023deeponet}
\bibinfo{author}{Moya, C.}, \bibinfo{author}{Zhang, S.}, \bibinfo{author}{Lin,
  G.}, \bibinfo{author}{Yue, M.}, \bibinfo{year}{2023}.
\newblock \bibinfo{title}{Deeponet-grid-uq: A trustworthy deep operator
  framework for predicting the power grid’s post-fault trajectories}.
\newblock \bibinfo{journal}{Neurocomputing} \bibinfo{volume}{535},
  \bibinfo{pages}{166--182}.
\bibitem[{Nagabandi et~al.(2020)Nagabandi, Konolige, Levine and
  Kumar}]{nagabandi2020deep}
\bibinfo{author}{Nagabandi, A.}, \bibinfo{author}{Konolige, K.},
  \bibinfo{author}{Levine, S.}, \bibinfo{author}{Kumar, V.},
  \bibinfo{year}{2020}.
\newblock \bibinfo{title}{Deep dynamics models for learning dexterous
  manipulation}, in: \bibinfo{booktitle}{Conference on Robot Learning},
  \bibinfo{organization}{PMLR}. pp. \bibinfo{pages}{1101--1112}.
\bibitem[{Proctor et~al.(2018)Proctor, Brunton and
  Kutz}]{proctor2018generalizing}
\bibinfo{author}{Proctor, J.L.}, \bibinfo{author}{Brunton, S.L.},
  \bibinfo{author}{Kutz, J.N.}, \bibinfo{year}{2018}.
\newblock \bibinfo{title}{Generalizing koopman theory to allow for inputs and
  control}.
\newblock \bibinfo{journal}{SIAM Journal on Applied Dynamical Systems}
  \bibinfo{volume}{17}, \bibinfo{pages}{909--930}.
\bibitem[{Qin et~al.(2021)Qin, Chen, Jakeman and Xiu}]{qin2021data}
\bibinfo{author}{Qin, T.}, \bibinfo{author}{Chen, Z.},
  \bibinfo{author}{Jakeman, J.D.}, \bibinfo{author}{Xiu, D.},
  \bibinfo{year}{2021}.
\newblock \bibinfo{title}{Data-driven learning of nonautonomous systems}.
\newblock \bibinfo{journal}{SIAM Journal on Scientific Computing}
  \bibinfo{volume}{43}, \bibinfo{pages}{A1607--A1624}.
\bibitem[{Qin et~al.(2019)Qin, Wu and Xiu}]{qin2019data}
\bibinfo{author}{Qin, T.}, \bibinfo{author}{Wu, K.}, \bibinfo{author}{Xiu, D.},
  \bibinfo{year}{2019}.
\newblock \bibinfo{title}{Data driven governing equations approximation using
  deep neural networks}.
\newblock \bibinfo{journal}{Journal of Computational Physics}
  \bibinfo{volume}{395}, \bibinfo{pages}{620--635}.
\bibitem[{Raissi et~al.(2018)Raissi, Perdikaris and
  Karniadakis}]{raissi2018multistep}
\bibinfo{author}{Raissi, M.}, \bibinfo{author}{Perdikaris, P.},
  \bibinfo{author}{Karniadakis, G.E.}, \bibinfo{year}{2018}.
\newblock \bibinfo{title}{Multistep neural networks for data-driven discovery
  of nonlinear dynamical systems}.
\newblock \bibinfo{journal}{arXiv preprint arXiv:1801.01236} .
\bibitem[{Ranade et~al.(2021)Ranade, Gitushi and
  Echekki}]{ranade2021generalized}
\bibinfo{author}{Ranade, R.}, \bibinfo{author}{Gitushi, K.},
  \bibinfo{author}{Echekki, T.}, \bibinfo{year}{2021}.
\newblock \bibinfo{title}{Generalized joint probability density function
  formulation inturbulent combustion using deeponet}.
\newblock \bibinfo{journal}{arXiv preprint arXiv:2104.01996} .
\bibitem[{Schaeffer(2017)}]{schaeffer2017learning}
\bibinfo{author}{Schaeffer, H.}, \bibinfo{year}{2017}.
\newblock \bibinfo{title}{Learning partial differential equations via data
  discovery and sparse optimization}.
\newblock \bibinfo{journal}{Proceedings of the Royal Society A: Mathematical,
  Physical and Engineering Sciences} \bibinfo{volume}{473},
  \bibinfo{pages}{20160446}.
\bibitem[{Sun et~al.(2020)Sun, Zhang and Schaeffer}]{sun2020neupde}
\bibinfo{author}{Sun, Y.}, \bibinfo{author}{Zhang, L.},
  \bibinfo{author}{Schaeffer, H.}, \bibinfo{year}{2020}.
\newblock \bibinfo{title}{Neupde: Neural network based ordinary and partial
  differential equations for modeling time-dependent data}, in:
  \bibinfo{booktitle}{Mathematical and Scientific Machine Learning},
  \bibinfo{organization}{PMLR}. pp. \bibinfo{pages}{352--372}.
\bibitem[{Sutton(1990)}]{sutton1990integrated}
\bibinfo{author}{Sutton, R.S.}, \bibinfo{year}{1990}.
\newblock \bibinfo{title}{Integrated architectures for learning, planning, and
  reacting based on approximating dynamic programming}, in:
  \bibinfo{booktitle}{Machine learning proceedings 1990}.
  \bibinfo{publisher}{Elsevier}, pp. \bibinfo{pages}{216--224}.
\bibitem[{Sutton and Barto(2018)}]{sutton2018reinforcement}
\bibinfo{author}{Sutton, R.S.}, \bibinfo{author}{Barto, A.G.},
  \bibinfo{year}{2018}.
\newblock \bibinfo{title}{Reinforcement learning: An introduction}.
\newblock \bibinfo{publisher}{MIT press}.
\bibitem[{Wang et~al.(2021a)Wang, Teng and Perdikaris}]{wang2021understanding}
\bibinfo{author}{Wang, S.}, \bibinfo{author}{Teng, Y.},
  \bibinfo{author}{Perdikaris, P.}, \bibinfo{year}{2021}a.
\newblock \bibinfo{title}{Understanding and mitigating gradient flow
  pathologies in physics-informed neural networks}.
\newblock \bibinfo{journal}{SIAM Journal on Scientific Computing}
  \bibinfo{volume}{43}, \bibinfo{pages}{A3055--A3081}.
\bibitem[{Wang et~al.(2021b)Wang, Wang and Perdikaris}]{wang2021learning}
\bibinfo{author}{Wang, S.}, \bibinfo{author}{Wang, H.},
  \bibinfo{author}{Perdikaris, P.}, \bibinfo{year}{2021}b.
\newblock \bibinfo{title}{Learning the solution operator of parametric partial
  differential equations with physics-informed deeponets}.
\newblock \bibinfo{journal}{Science advances} \bibinfo{volume}{7},
  \bibinfo{pages}{eabi8605}.
\bibitem[{Wang et~al.(2019)Wang, Bao, Clavera, Hoang, Wen, Langlois, Zhang,
  Zhang, Abbeel and Ba}]{wang2019benchmarking}
\bibinfo{author}{Wang, T.}, \bibinfo{author}{Bao, X.},
  \bibinfo{author}{Clavera, I.}, \bibinfo{author}{Hoang, J.},
  \bibinfo{author}{Wen, Y.}, \bibinfo{author}{Langlois, E.},
  \bibinfo{author}{Zhang, S.}, \bibinfo{author}{Zhang, G.},
  \bibinfo{author}{Abbeel, P.}, \bibinfo{author}{Ba, J.}, \bibinfo{year}{2019}.
\newblock \bibinfo{title}{Benchmarking model-based reinforcement learning}.
\newblock \bibinfo{journal}{arXiv preprint arXiv:1907.02057} .
\bibitem[{Wittenmark et~al.(2002)Wittenmark, {\AA}str{\"o}m and
  {\AA}rz{\'e}n}]{wittenmark2002computer}
\bibinfo{author}{Wittenmark, B.}, \bibinfo{author}{{\AA}str{\"o}m, K.J.},
  \bibinfo{author}{{\AA}rz{\'e}n, K.E.}, \bibinfo{year}{2002}.
\newblock \bibinfo{title}{Computer control: An overview}.
\newblock \bibinfo{journal}{IFAC Professional Brief} \bibinfo{volume}{1},
  \bibinfo{pages}{2}.
\bibitem[{Zhang et~al.(2022)Zhang, Leung and Schaeffer}]{zhang2022belnet}
\bibinfo{author}{Zhang, Z.}, \bibinfo{author}{Leung, W.T.},
  \bibinfo{author}{Schaeffer, H.}, \bibinfo{year}{2022}.
\newblock \bibinfo{title}{Belnet: Basis enhanced learning, a mesh-free neural
  operator}.
\newblock \bibinfo{journal}{arXiv preprint arXiv:2212.07336} .

\end{thebibliography}

\end{document}